\documentclass[11pt]{article}
	
	\usepackage{algorithm}
	\usepackage{algpseudocode}%algorithm and pseudocode
    \usepackage{url}
    \usepackage{verbatim}
    \usepackage[titletoc]{appendix}
    \usepackage{graphicx}
	\usepackage{amsmath}
	\usepackage{amsthm}
	\usepackage{amsfonts}
	\usepackage{graphicx}
    \usepackage{nicefrac}
    \usepackage{longtable}
    \usepackage{color}
    \usepackage{graphicx, amssymb,graphics}
 	\usepackage{subfigure}
 	\usepackage{epstopdf}

\newtheorem{rem}{Remark}[section]
\newtheorem{theorem}{Theorem}[section]

\newtheorem{lem}{Lemma}[section]

\newcommand{\n}{\mbox{\boldmath$n$}}

\newcommand{\diff}{\mathrm{d}}
\newcommand{\dS}{\diff S}

\newcommand{\dx}{\mbox{d}\textbf{x}}
\newcommand\dt {{\Delta t}}

\newcommand{\mF}{\mathcal F}
\newcommand{\mJ}{\mathcal J}
\newcommand{\Cperx}{{\mathcal C}_{\rm per}^x}
\newcommand{\Cperxo}{\mathring{{\mathcal C}}_{\rm per}^x}

	\newcommand\be {\begin{equation}}
	\newcommand\ee {\end{equation}}

	\title{Convergence analysis of positivity-preserving finite difference scheme for the Flory-Huggins-Cahn-Hilliard equation with dynamical boundary condition}

	\date{\today}

\begin{document}
	
\author{
Yunzhuo Guo \thanks{Department of Applied Mathematics, Hong Kong Polytechnic University, Hung Hom, Hong Kong (yunzguo@polyu.edu.hk)}
\and Cheng Wang\thanks{Department of Mathematics, The University of Massachusetts, North Dartmouth, MA  02747, USA (cwang1@umassd.edu)}
\and	
Zhengru Zhang\thanks{Corresponding author. Laboratory of Mathematics and Complex Systems, Beijing Normal University, Beijing 100875, P.R. China (zrzhang@bnu.edu.cn)}
%\and	
%Cheng Wang\thanks{Department of Mathematics, The University of Massachusetts, North %Dartmouth, MA  02747, USA (cwang1@umassd.edu)}
%\and
%Cheng Wang\thanks{Department of Mathematics, The University of Massachusetts, North Dartmouth, MA  02747, USA (Corresponding Author: cwang1@umassd.edu)}
%	\and
%Steven M. Wise\thanks{Department of Mathematics, The University of Tennessee, Knoxville, TN 37996, USA (swise1@utk.edu)}
%\and	
%Zhengru Zhang\thanks{Laboratory of Mathematics and Complex Systems, Beijing Normal University, Beijing 100875, P.R. China (zrzhang@bnu.edu.cn)}
}

 	\maketitle
	\numberwithin{equation}{section}
	
\begin{abstract}	 
The Cahn-Hilliard equation has a wide range of applications in many areas of physics and chemistry. To describe the short-range interaction between the solution and the boundary, scientists have constructed dynamical boundary conditions by introducing boundary energy. In this work, %the model raised by Liu and Wu is considered. 
the dynamical boundary condition is located on two opposite edges of a square domain and is connected with bulk by a normal derivative. A convex-splitting numerical approach is proposed to enforce the positivity-preservation and energy dissipation, combined with the finite difference spatial approximation. The $\ell^\infty (0, T; H_h^{-1}) \cap \ell^2 (0, T; H_h^1)$ convergence analysis and error estimate is theoretically established, with the first order accuracy in time and second order accuracy in space. The bulk and surface discrete mass conservation of the exact solution is required to reach the mean-zero property of the error function, so that the associated discrete $H_h^{-1}$ norm is well-defined. The mass conservation on the physical boundary is maintained by the classic Fourier projection. In terms of the mass conservation in bulk, we introduce a trigonometric auxiliary function based on the truncation error expansion, so that the bulk mass conservation is achieved, and it has no effect on the boundary. The smoothness of trigonometric function makes the Taylor expansion valid and maintains the convergence order of truncation error as well. As a result, the convergence analysis could be derived with a careful nonlinear error estimate. %As an alternative, this correction method can also be extended to the analysis of various boundary conditions and complex problems.

\bigskip

\noindent
{\bf Key words and phrases}: 
Cahn-Hilliard equation, Flory-Huggins energy potential, dynamical boundary condition,  convergence analysis
%Poisson-Nernst-Planck (PNP) system, logarithmic energy potential, positivity preserving, optimal rate convergence analysis, modified Newton iteration, linear convergence rate 
%	\begin{AMS}

\noindent
{\bf AMS subject classification}: \, 35K35, 35K55, 49J40, 65M06, 65M12	
\end{abstract}
	
%{\bf keywords.} 

%Bao2021a, Bao2021b, Knopf2021, LiuC2019, Metzger2021

\section{Introduction} 
The Cahn-Hilliard (CH) equation plays an important role in material science and biological applications. It was first constructed by Cahn and Hilliard \cite{cahn58} to describe the free energy of isotropic systems with non-uniform density, and is widely used in various problems. Although traditional boundary conditions, such as periodic and Neumann type ones, bring convenience to numerical analysis and practical computation, they may not be applicable to some particular problems like the moving contact line model and the interaction near the solid wall. A simple example is that Neumann boundary condition makes the interface always perpendicular to the solid wall, which is unphysical in most systems. To resolve this issue, scientists have proposed various dynamical boundary conditions. The one considered in this work is referred as the Liu-Wu model \cite{LiuC2019}, based on energy variation. In this model, a surface energy functional is introduced to describe the short-range interaction between the solution and the solid boundary. A theoretical analysis of both the weak and strong solutions, including the existence, uniqueness and corresponding regularity, was discussed in \cite{LiuC2019}. The bulk and surface energies are given by
\begin{align} 
& E_{bulk} (\phi) = \int_\Omega \frac{\varepsilon^2}{2} | \nabla \phi |^2 + F (\phi) \ \dx , \quad E_{surf} (\phi) = \int_\Gamma \frac{\varepsilon\kappa}{2} | \nabla_\Gamma \psi |^2 + G (\psi)  \,  \dS, \label{energy-CH-1} \\
& E_{total} = E_{bulk} + E_{surf} 
\end{align} 
where $\varepsilon$ corresponds to interface thickness, $\kappa$ stands for the surface diffusion, $F$ and $G$ refer to the nonlinear double-well potential. When $\kappa=0$, the equation is reduced to the  moving contact line problem \cite{PRL-63-766}. The governing equation becomes 
\begin{align} 
  & 
  \phi_t = \Delta \mu ,  \quad 
  \mu = F' (\phi) - \varepsilon^2 \Delta \phi,  \quad\quad\quad\quad\quad\quad\quad\quad  \mbox{in $\Omega$},  \label{equation-CHDBC-1} \\
 & 
 \partial_{\n} \mu = 0 ,  \quad \phi|_\Gamma = \psi ,  \quad 
 \quad\quad\quad\quad\quad\quad\quad\quad\quad\quad\quad\  \mbox{on $\partial \Omega$},  \label{equation-CHDBC-2}   \\
& 
  \psi_t = \Delta_\Gamma \mu_\Gamma ,  \quad 
  \mu_\Gamma = - \varepsilon\kappa \Delta_\Gamma \psi +  G' (\psi) 
  + \varepsilon^2 \partial_{\n} \phi \quad \mbox{on $\partial \Omega$} \label{equation-CHDBC-3}    
\end{align}  
in which $\mu_\Gamma$ stands for the surface chemical potential (instead of the boundary projection of $\mu$), and $\Delta_\Gamma$ is the Laplace-Beltrami operator on the Lipschitz continuous boundary $\Gamma:=\partial \Omega$. Meanwhile, the normal derivative term couples the interior and surface. An energy dissipation law becomes available through integration by parts:
\begin{align}
	\frac{d}{dt} E_{total}(\phi) & = \int_\Omega\left\{-\varepsilon^2 \Delta \phi  + F'(\phi)  \right\} \partial_t\phi  \,\dx + \int_\Gamma \varepsilon^2\partial_n\phi \partial_t\phi \, \dS \nonumber
	\\
	& \quad + \int_\Gamma \left\{-\kappa\varepsilon \Delta_\Gamma\psi   + G'(\psi)  \right\} \partial_t\psi \,\dS \nonumber
	\\
	& = \int_\Omega\left\{-\varepsilon^2 \Delta \phi  + F'(\phi)  \right\} \partial_t\phi  \,\dx  \label{energy}
	\\
	& \quad + \int_\Gamma \left\{-\kappa\varepsilon \Delta_\Gamma\psi   + G'(\psi) +\varepsilon^2\partial_n\phi \right\} \partial_t\psi \,\dS \nonumber
	\\
	& = -\int_\Omega\nabla\mu\cdot \nabla\mu  \,\dx - \int_\Gamma \nabla_\Gamma \mu_\Gamma \cdot \nabla_\Gamma \mu_\Gamma  \,\dS. \nonumber
\end{align}
Therefore, one obtains 
\begin{equation}
	\frac{d}{dt} E_{total} = - \| \nabla \mu \|^2 - \| \nabla_\Gamma \mu_\Gamma \|_\Gamma^2 \le 0 .
	\label{energy dissipation-3} 
\end{equation}  

If the Flory-Huggins logarithmic potential is chosen, $F$ and $G$ are given by 
\begin{align} 
  & F (\phi) = (1+ \phi) \ln (1+\phi) + (1-\phi) \ln (1-\phi) - \frac{\theta_0}{2} \phi^2 , \label{log energy 1}\\
  & G (\psi) = (1+ \psi) \ln (1+\psi) + (1-\psi) \ln (1-\psi) - \frac{\theta_0}{2} \psi^2 . \label{log energy 2}
\end{align} 
The scientific difficulty in the choice of the logarithmic function is associated with the positivity-preserving issue, which comes from the singularity. As an approximation, the following non-singular polynomial free energy is also wildly studied in various phase field problems: 
\begin{equation}
    F(\phi)=\frac{1}{4}(\phi^2-1)^2. \label{polynomial energy 1}
\end{equation}

In the Liu-Wu model, the dynamical boundary condition is given by a lower dimensional CH-type equation coupled with interior equation through normal derivative $\partial_{\n} \phi$. One major feature is that no material exchange is allowed between the bulk and the boundary, therefore each of both satisfies a mass conservation. In other type of dynamical boundary conditions, the situation may be different. Gal et al. derived a set of dynamical boundary conditions for the CH equation in \cite{Gal06} (Gal model). The dynamical equation of the mass fraction on the surface follows Allen-Cahn equation with an additional flux, so that neither total mass nor the boundary mass is conservative. Goldstein et al. \cite{Goldstein11} modified the boundary transport between bulk and surface and proposed the GMS model, which allows for equal mass exchange between surface and bulk and ensures the total mass conservation only. 

Three basic dynamic boundary conditions are widely used in various models, based on the conservation properties. As a foundation, scientists have generalized the basic models and attempt to establish a unified form of boundary conditions. In 2019, Knopf and Lam presented another set of dynamical boundary conditions for the CH equation by extending the Liu-Wu model at the boundary (KL model) \cite{Knopf_2020}. Furthermore, a fairly general formulation named KLLM model was raised in \cite{Knopf2021} that includes the Liu-Wu model and GMS model in the sense of limit. Jing and Wang raised a thermodynamically consistent dynamical boundary conditions following the generalized Onsager principle \cite{Jing2023}. This model not only covers multiple basic boundary conditions, but also can be extended to non-local problems. 

In a very recent work \cite{CH-DBC-2024}, the boundary condition was reduced to half-dynamical and half-periodic to adapt a square domain, i.e. the dynamical boundary condition in the $y$-direction and periodic boundary condition in the $x$-direction or vice versa. With a period length $\tau > 0$ and physical boundary $y=0,1$, the PDE system becomes 
\begin{align}
	& \partial_t \phi=\Delta \mu, \quad \mu= F^{\prime}(\phi)-\varepsilon^2 \Delta \phi, \label{DBC-Periodical-1}   \\
	& \phi(x+a \tau, y)=\phi(x, y), \quad \mu(x+a \tau, y)=\mu(x, y), \quad a \in \mathbb{Z}, 
	\label{DBC-Periodical-2}
	\\
	& \partial_{\boldsymbol{n}} \mu|_{y=0,1}=0,\quad \phi|_{y=0}=\psi^B,\quad \phi|_{y=1}=\psi^T,
	\label{DBC-Periodical-3} 
	\\
& \psi_t^B=D_x^2 \mu_B, \quad \mu_B = G^{\prime}\left(\psi^B\right)-\kappa\varepsilon D_x^2 \psi^B+\left.\varepsilon^2 \partial_{\boldsymbol{n}} \phi\right|_{y=0},
	\label{DBC-Periodical-4} 
	\\
& \psi_t^T=D_x^2 \mu_T, \quad \mu_T= G^{\prime}\left(\psi^T\right)-\kappa\varepsilon D_x^2 \psi^T+\left.\varepsilon^2 \partial_{\boldsymbol{n}} \phi\right|_{y=1},
	\label{DBC-Periodical-5}
	\\
& \psi^T(x+a\tau) = \psi^T(x), \quad \mu^T(x+a\tau) = \mu^T(x), \quad a\in\mathbb{Z}, 
	\label{DBC-Periodical-6}
	\\
& \psi^B(x+a\tau) = \psi^B(x), \quad \mu^B(x+a\tau) = \mu^B(x), \quad a\in\mathbb{Z}.
	\label{DBC-Periodical-7}
\end{align}
The surface energy reduces to two parts:
\begin{align}
	E_{\rm surf} =& \int_{\{y=0\}} \left(  G(\psi^B) + \frac{\kappa\varepsilon}{2} | \nabla_\Gamma \psi^B |^2 \right) \dS + \int_{\{y=1\}} \left( G(\psi^T) + \frac{\kappa\varepsilon}{2} | \nabla_\Gamma \psi^T |^2 \right) \dS.  \nonumber
\end{align}
At the physics level, the simplified system contains two parallel solid walls, with the solution confined between them. As demonstrated in \cite{Cherfils_10}, the modification can be mathematically described by a quotient domain that 
$$
\Omega=\Pi_{i=1}^{d-1}\left(\mathbb{R} /\left(L_i \mathbb{Z}\right)\right) \times\left(0, L_d\right), \quad L_i>0, i=1, \ldots, d, \quad d=2 \text { or } 3, 
$$
with a boundary
$$
\Gamma=\partial \Omega=\Pi_{i=1}^{d-1}\left(\mathbb{R} /\left(L_i \mathbb{Z}\right)\right) \times\left\{0, L_d\right\}.
$$
In \cite{CH-DBC-2024}, a convex-splitting finite difference scheme was implemented for \eqref{DBC-Periodical-1}-\eqref{DBC-Periodical-7}. The unique solvability, positivity-preservation and energy dissipation were proved, while the convergence analysis has not been reported. In this work, the authors aim to provide a detailed convergence analysis of the proposed numerical scheme. 

Plenty of numerical studies have been available for the dynamical boundary conditions \cite{Cherfils_10, Israel2015, Kenzler2001139, Liu2024, Metzger2021,Nabet2014}, and the convergence analysis of semi-discrete temporal discretization has been reported as well \cite{Bao2021a,Bao2021b,Meng23}. Meanwhile, an optimal rate convergence analysis for the fully discrete scheme remains an open problem, to the best of our knowledge. In fact, one essential difficulty is the discrete mass conservation of the exact solution. Such a discrete mass conservation is necessary in the theoretical analysis of convergence estimate, since a discrete $H^{-1}$ error estimate requires a mean-zero property of the numerical error function. This property is automatically satisfied for the continuous PDE solution and semi-discrete numerical solution, while it is invalid for the fully discrete numerical method. The dismatch comes from the spatial truncation error. If the classic periodic boundary condition or homogeneous Neumann boundary condition is taken, the Fourier projection (including cosine and sine projections) turns out to be a useful choice, since both the continuous and discrete integrals of a trigonometric function vanish over a single period. If the solution has sufficient regularity, a spectral approximation accuracy becomes valid between the projection and the exact solution. In turn, such a projection process would not affect the consistency analysis. This technique has been widely used in the finite difference analysis, such as the phase field crystal model \cite{baskaran13b, dong18a}, the CH equation and its variants~\cite{chen22a, chen19b, Dong2021a, Dong2022a, dong19b, GuoY24} and the classic Poisson-Nernst-Planck equation~\cite{LiuC2021a}, etc. If the finite element spatial discretization is taken, a similar projection technique has also been applied in the convergence analysis, with a proper choice of function space~\cite{diegel17, diegel16}. Moreover, for the gradient flow equation with a non-singular free energy, the Fourier projection could be replaced by a careful calculation of the discrete mass and the chemical potential \cite{chen16, guan17a, guo16, guo2021}. 

On the other hand, if the dynamical boundary condition is taken, the Fourier projection approach does not work out any more. To overcome this difficulty, we carry out the theoretical analysis in two steps. First, the Fourier projection is applied only in the periodic $x$-direction. Second, a trigonometric auxiliary function is explicitly defined to deal with the $y$-direction. The first step aims to preserve the discrete mass conservation on the physical boundary. The latter is used to preserve the discrete bulk mass conservation without affecting the surface part. The introduced auxiliary function should have sufficient regularity, so that the spatial discretization and Taylor expansion could be carried out in a straightforward way. This is the reason why the construction of the approximate solution in the consistency analysis is based on the smooth trigonometric functions. In addition, it is also proved that the auxiliary function is an infinitesimal perturbation independent of the time discretization and with the same order as the spatial truncation error, so that this technique could be expanded to higher-order convergence and other kinds of boundary conditions in a square domain. For the logarithmic singular term, its convexity leads to a non-negative inner product in the associated error estimate. In turn, the singularity could not cause any additional difficulty in the $H^{-1}$-convergence analysis. In fact, this theoretical framework is also valid for an error estimate in a higher order norm, such as the $H^1$ one. In that case, thanks to the smoothness of the auxiliary function, the higher-order convergence analysis could be accomplished with the help of the rough and refined error estimates. 

The rest of this paper is organized as follows. In Section \ref{sec:numerical scheme}, the finite difference spatial discretization is briefly reviewed. The fully discrete convex-splitting numerical scheme and the associating physical structure-preserving properties will also be stated without proof. Subsequently, the detailed $\ell^\infty (0, T; H_h^{-1}) \cap \ell^2 (0, T; H_h^1)$ convergence analysis will be provided in Section \ref{sec:convergence analysis}. In addition, the second order BDF2 numerical scheme is outlined and analyzed in Section \ref{sec: 2nd numerical scheme}. Some concluding remarks are made in Section \ref{sec:conclusion}.

\section{The first order numerical scheme} \label{sec:numerical scheme}	
A first order accurate, fully discrete numerical scheme will be reviewed in this section. The associated structure-preserving properties will be stated as well. The detailed proof can be found in \cite{CH-DBC-2024}. The centered finite difference spatial discretization is used. %We first briefly review some of the basics of this methodology.

\subsection{A brief description of the spatial discretization}
The computational domain is taken as $\Omega=(0,1)^2$, and an extension to a three-dimensional domain will be straightforward. The physical boundary condition is set at the top and bottom boundary sections of $\Omega$, i.e. $y=0,1$. The case of physical boundary conditions on all four boundary sections could be analyzed in a similar manner, with a proper approximation at four vertices. A uniform spatial mesh size $\Delta x = \Delta y = h= \frac{1}{N}$ with $N \in \mathbb{N}$ is assumed, for simplicity of presentation. In particular, $f_{i,j}$ stands for the numerical value of $f$ at the mesh points $((i+\frac{1}{2})h, jh)$. In more detail, a cell-centered mesh in the $x$-direction is taken, combined with the regular mesh points in the $y$-direction. We denote a space 
$$\Cperx:=\{ (f_{i,j}) | f_{i+\alpha N,j}=f_{i,j}, \ i=0,\cdots,N-1,\ j=0,\cdots,N,\ \ \forall \alpha \in \mathbb{Z} \} , $$ 
with the discrete periodic boundary condition imposed in the $x$-direction. In addition, for any two grid functions $f,g \in \Cperx$, the discrete $\ell^2$ inner product and the associated $\ell^2$ norm are defined as 
\begin{equation}                          
    (f,g):=h^2 \sum_{i=0}^{N-1} \sum_{j=0}^{N} \omega_j f_{i,j}g_{i,j},\quad w_j = \left\{ 
 	\begin{array}{l} 
        1 , \, \, \, 1 \le j \le N-1 , \\ 
       \frac{1}{2} , \, \, \, j=0, N , 
	\end{array} \right. \quad \|f\|_2^2 := (f,f).
\end{equation}
The mean-zero function space is introduced, based on this inner product: 
$$
\Cperxo:=\{ f \in \Cperx | (f,1)=0 \}.
$$
Moreover, the normal derivative is defined as
\begin{align}
    \tilde{D}_y f_{i,0} = \frac{f_{i,1}-f_{i,-1}}{2h}, \quad \tilde{D}_y f_{i,N} = \frac{f_{i,N+1}-f_{i,N-1}}{2h}, \quad \forall f \in \Cperx , 
\end{align}
and the homogeneous Neumann boundary condition in the $y$-direction becomes $\tilde{D}_y f_{\cdot,0} = \tilde{D}_y f_{\cdot,N}=0$. Of course, the discretization involving the ghost points implies a condition that the function must be continuous on the boundary. This is an important assumption in the numerical calculation and the convergence analysis.

For the vector functions $\boldsymbol{f} = (f^x,f^y)^T$ and $\boldsymbol{g} = (g^x,g^y)^T$, with $f^x_{i+\frac{1}{2},j},g^x_{i+\frac{1}{2},j}$ evaluated at $((i+1)h, jh)$ and $f^y_{i,j+\frac{1}{2}},g^y_{i,j+\frac{1}{2}}$ evaluated at $((i+\frac{1}{2})h, (j+\frac{1}{2})h)$, respectively, the corresponding inner product is defined as 
\begin{align}
    &(\boldsymbol{f},\boldsymbol{g}):=(f^x,g^x)_x+(f^y,g^y)_y, \\
    &(f^x,g^x)_x:= (a_x(f^x g^x),1), \quad  (f^y,g^y)_y:=h^2\sum_{i,j=0}^{N-1} f^y_{i,j+\frac{1}{2}}g^y_{i,j+\frac{1}{2}} , 
\end{align}
where $a_x$ is the average operator given by $a_xf^x_{i,j}:=\frac{1}{2}(f^x_{i+\frac{1}{2},j}+f^x_{i-\frac{1}{2},j})$ and $a_y$ is defined in the same way. Based on the above definition, the summation-by-part formula holds for $\psi, \phi \in \Cperx$:
\begin{align} 
(\psi, \Delta_h \phi) = -(\nabla_h \psi, \nabla_h \phi) + (\tilde{D}_y\phi_{\cdot , N}, \psi_{\cdot , N})_\Gamma - (\tilde{D}_y\phi_{\cdot, 0}, \psi_{\cdot, 0})_\Gamma , 
\end{align}
where the boundary part on the right hand side vanishes if $\phi$ satisfies the homogeneous Neumann boundary condition. Meanwhile, $(\cdot,\cdot)_\Gamma$ is the one-dimensional inner product, defined as
\begin{equation}
    (f,g)_\Gamma := h \sum_{i,j=0}^{N-1} f_ig_i,\quad \| f \|_{2,\Gamma}^2=(f,f)_\Gamma.
\end{equation}
In addition, a modified negative Laplacian operator $L_h$ is defined. This definition differs from the standard discrete negative Laplacian, in order to incorporate the discrete homogeneous Neumann boundary conditions. Specifically, $L_h : \Cperx \rightarrow \Cperxo$, is given by 
	\begin{equation}
L_h \psi_{i,j} = 
	\begin{cases}
- D_x^2 \psi_{i,0} - 2\frac{\psi_{i,1} - \psi_{i,0}}{h^2}, & j = 0,
	\\
- D_x^2 \psi_{i,N} - 2\frac{\psi_{i,N-1}-\psi_{i,N}}{h^2}, & j = N,
	\\
-\Delta_h \psi_{i,j}, & \mbox{otherwise}.
	\end{cases}
	\end{equation}
Under the homogeneous Neumann boundary condition, $L_h$ is equivalent to the standard negative Laplacian operator, and the summation-by-part formula is valid for $L_h$: 
\begin{equation}
(\psi, L_h \phi) = (\nabla_h \psi, \nabla_h \phi), \quad  \forall \psi, \phi \in \Cperx.
\end{equation}
On the other hand, if $L_h$ is restricted to domain $\Cperxo$, it becomes a bijection. The discrete $H^{-1}$ norm is defined based on this property: 
\begin{equation}
	\| \phi \|_{-1} := ( \phi, \psi ) , \quad \mbox{for} \, \, \, \phi \in \Cperxo , 
\end{equation}
where $\psi \in \Cperxo$ is the unique solution of the equation $L_h \psi = \phi$.

\subsection{The first order convex-splitting scheme}
Based on the convex-splitting approach in both bulk and boundary parts, respectively, the following finite difference scheme is proposed: given $\phi^n \in {\mathcal C}_{\rm per}^x$, we solve nonlinear equations to find  $\phi^{n+1}, \mu^{n+1}\in {\mathcal C}_{\rm per}^x$, such that
\begin{align}  
& \frac{\phi^{n+1} - \phi^n}{\dt} = \Delta_h \mu^{n+1} , \label{scheme-CHDBC-1} \\
& \mu^{n+1} = \ln ( 1+ \phi^{n+1} ) - \ln (1 - \phi^{n+1} ) - \theta_0 \phi^n  - \varepsilon^2 \Delta_h \phi^{n+1}, \label{scheme-CHDBC-1.1} \\
& \tilde{D}_{y} \mu^{n+1}_{i, 0} = \tilde{D}_{y} \mu^{n+1}_{i, N} = 0  ,  \quad 
  \phi^{n+1}_{i,0} = \phi^{B,n+1}_{i} , \, \, \, \phi^{n+1}_{i,N} = \phi^{T,n+1}_{i} , \label{scheme-CHDBC-2}	\\
& \frac{\phi^{B,n+1} - \phi^{B,n}}{\dt} = D_x^2 \mu^{n+1}_B, \, \, \, 
  \frac{\phi^{T,n+1} - \phi^{T,n}}{\dt} = D_x^2 \mu^{n+1}_T , \label{scheme-CHDBC-3}	\\
& \mu^{n+1}_B = \ln ( 1+ \phi^{B,n+1} ) - \ln (1 - \phi^{B,n+1} ) - \theta_0 \phi^{B,n} - \kappa \varepsilon D_x^2 \phi^{B,n+1} - \varepsilon^2 \tilde{D}_{y} \phi^{n+1}_{\cdot, 0} ,\label{scheme-CHDBC-4}\\
& \mu^{n+1}_T = \ln ( 1+ \phi^{T,n+1} ) - \ln (1 - \phi^{T,n+1} ) - \theta_0 \phi^{T,n} - \kappa \varepsilon D_x^2 \phi^{T,n+1} + \varepsilon^2 \tilde{D}_{y} \phi^{n+1}_{\cdot, N}. \label{scheme-CHDBC-5}  
\end{align}
The numerical scheme obviously satisfies the bulk and surface mass conservation, which comes from the summation-by-part formula under the homogeneous Neumann and periodic boundary conditions. The positivity-preservation and free energy dissipation law are stated below. 

\begin{theorem}  \cite{CH-DBC-2024} \label{CHDBC-positivity} 
Given $\phi^n \in \Cperx$, with $-1 < \phi^n_{i,j} < 1$, $0 \le i, j \le N$,  and $\overline{\phi^n} = \beta_0$, $\overline{\phi^{B,n}}^\Gamma  = \beta_{B,0}$, $\overline{\phi^{T,n}}^\Gamma  = \beta_{T,0}$, there exists a unique solution $\phi^{n+1} \in \Cperx$ to scheme \eqref{scheme-CHDBC-1}-\eqref{scheme-CHDBC-5}, with $-1 < \phi^{n+1}_{i,j} < 1$, $0 \le i \le N-1$, $0 \le j \le N$, and $\overline{\phi^{n+1}} = \beta_0$,  $\overline{\phi^{B,n+1}}^\Gamma  = \beta_{B,0}$, $\overline{\phi^{T,n+1}}^\Gamma  = \beta_{T,0}$.
\end{theorem} 

In particular, the numerical solution of \eqref{scheme-CHDBC-1}-\eqref{scheme-CHDBC-5} is equivalent to the unique minimizer and stationary point of the convex functional $\mF_h^n$, given by
\begin{align}
    \mF_h^n (\phi) & :=  \frac{1}{2\dt} \|  \phi - \phi^n \|_{-1}^2 + \frac{1}{h \dt } \|  \phi^B - \phi^{B,n} \|_{-1, \Gamma}^2 + \frac{1}{h \dt } \|  \phi^T - \phi^{T,n} \|_{-1,\Gamma}^2  \nonumber \\
& \quad +  (I(\phi),1) + \frac{2}{h}  (I(\phi^B),1)_\Gamma  + \frac{2}{h} (I(\phi^T),1)_\Gamma  \nonumber  \\
& \quad  + \frac{\varepsilon^2}{2}(\phi, L_h\phi)  - \frac{\kappa\varepsilon}{h}(\phi^B, D_x^2\phi^B)_\Gamma - \frac{\kappa\varepsilon}{h} (\phi^T, D_x^2\phi^T)_\Gamma \nonumber \\
& \quad - \theta_0 \Big( (\phi^n,\phi) + \frac{2}{h} (\phi^{B,n},\phi^B)_\Gamma + \frac{2}{h} (\phi^{T,n},\phi^T)_\Gamma \Big) \nonumber 
\end{align}
where
\begin{equation}
    I(\phi)=(1+\phi)\ln(1+\phi)+(1-\phi)\ln(1-\phi). \label{function I} 
\end{equation}
The admissible set is given as
\begin{equation}
A_h :=  \left\{  \phi \in \Cperx \ \middle| \ -1 < \phi_{i,j}  < 1 , \ 0 \le j \le N , \ i\in\mathbb{Z}, \ \overline{\phi} = \beta_0, \ \overline{\phi_B}  = \beta_{B,0}, \ \overline{\phi_T}  = \beta_{T,0} \right\}. \nonumber
\end{equation} 
Of course, we have a representation $\phi^{n+1} = \mathop{\mathrm{argmin}} \mF_h^n(\phi)$. 

The detailed proof of Theorem \ref{CHDBC-positivity} could be found in \cite{CH-DBC-2024}. It is noticed that the coefficient $2/h$ in the boundary terms comes from the difference between operator $L_h$ and the negative Laplacian one. Taking the top boundary as an example:
\begin{align}
-\varepsilon^2 \Delta_h \phi^{n+1}_{\cdot, N} & = \varepsilon^2 L_h \phi^{n+1}_{\cdot, N} - \frac{2\varepsilon^2}{h} \tilde{D}_{y} \phi^{n+1}_{\cdot, N}   \\
& = \varepsilon^2 L_h \phi^{n+1}_{\cdot, N} + \frac{2}{h} \Big( (-D_x^2)^{-1}(\frac{\phi^{T,n+1}-\phi^{T,n}}{\dt}) + I'(\phi^{T,n+1}) - \theta_0 \phi^{T,n} - \kappa\varepsilon D_x^2 \phi^{T,n+1} \Big). \nonumber
\end{align}
The additional normal derivative is replaced by equations \eqref{scheme-CHDBC-3} and \eqref{scheme-CHDBC-5}. The long stencil definition of the normal derivative yields the coefficient.

\begin{theorem} \cite{CH-DBC-2024} \label{CHDBC-energy} 
    Introduce a discrete energy as 
\begin{align} 
E_h (\phi) & := (I(\phi),1) + (I(\phi^B),1)_\Gamma + (I(\phi^T),1)_\Gamma - \frac{\theta_0}{2} \Big( \| \phi \|_2^2 + \| \phi_B \|_{2,\Gamma}^2 + \| \phi_T \|_{2,\Gamma}^2  \Big) \nonumber \\
& \quad\quad  -  \frac{\varepsilon^2}{2} \| \nabla_h \phi \|_2^2 - \frac{\kappa\varepsilon}{2} \Big( (\phi^B, D_x^2\phi^B)_\Gamma 
  + (\phi^T, D_x^2\phi^T)_\Gamma  \Big). \label{CHDBC-discrete energy}
\end{align} 
For any time step size $\dt>0$, the numerical solution of \eqref{scheme-CHDBC-1}-\eqref{scheme-CHDBC-5} satisfies the energy dissipation law
	\begin{equation} 
  E_h (\phi^{n+1} ) 
+ \dt  ( \| \nabla_h \mu^{n+1} \|_{2}^2 +  \| D_x \mu_B^{n+1} \|_{2,\Gamma}^2 +\| D_x \mu_T^{n+1} \|_{2,\Gamma}^2 ) \le E_h (\phi^n)   ,   
	\label{CHDBC-energy-0} 
	\end{equation} 
so that $E_h (\phi^n) \le E_h (\phi^0)$, for all $n \in \mathbb{N}_+$. 
\end{theorem}

\section{Convergence analysis and error estimate}  \label{sec:convergence analysis} 
Now we proceed into the convergence analysis. Let $\Phi$ be the exact PDE solution for the CH system \eqref{equation-CHDBC-1}-\eqref{equation-CHDBC-3}. For convenience, the superscript $B$ and $T$ denote the boundary projection to the bottom and top boundary sections, respectively. For example, $\Phi^T$ and $\Phi^B$ stand for the boundary parts of $\Phi$, and $\mu_T$, $\mu_B$ are not the boundary projections of $\mu$. With sufficiently regular initial data, it is reasonable to assume that the exact solution has a regularity of class $\mathcal{R}$:
$$
\Phi \in \mathcal{R}:=H^2\left(0, T ; C_{\text {per,x}}(\overline{\Omega})\right) \cap H^1\left(0, T ; C_{\text {per,x}}^2(\overline{\Omega})\right) \cap L^{\infty}\left(0, T ; C_{\text {per,x}}^6(\overline{\Omega})\right) , 
$$
where $C_{\text {per,x}}$ refers to the continuous function with periodic boundary condition in the $x$-direction. In addition, a separation property is assumed for the exact solution:
\begin{equation} 
1 + \Phi  \ge \epsilon_0 ,\quad  1 - \Phi \ge \epsilon_0 ,  \quad 0 \leq t \leq T, \quad  \mbox{for some $\epsilon_0 > 0$, at a point-wise level} . \label{separation property}
\end{equation}  

\subsection{Discrete mass conservation of the exact solution}

The $H^{-1}$ convergence analysis requires a discrete mean-zero property. The numerical solution unconditionally satisfies the discrete mass conservation, while the exact solution does not, because of the truncation error in the finite difference approximation. To overcome this difficulty, we need some a-priori treatments. Denote $\Phi_N(\cdot,y;t)=\mathcal{P}_x^N\Phi(\cdot,y;t)$, where $\mathcal{P}_x^N$ is the spatial Fourier projection (in the $x$-direction) to ${\cal B}^K$, the space of trigonometric polynomials of degree to and including $K$ (with $N=2K+1$). The following projection approximation is standard: if $\Phi \in L^\infty(0,T;H^\ell_{\rm per,x}(\Omega))$, $\ell\in\mathbb{N}$, we have 
\begin{equation} 
\| \Phi_N - \Phi\|_{L^\infty(0,T;H^k)} \leq C h^{\ell-k} \| \Phi \|_{L^\infty(0,T;H^\ell)} ,   
  \quad  \forall 0 \le k \le \ell . 
 \label{Fourier-approximation}
\end{equation}
In fact, the Fourier projection estimate does not automatically ensure the positivity of $1 + \Phi_N$ and $1 - \Phi_N$; on the other hand, one could enforce the separation property that $1 + \Phi_N \ge \frac{3}{4} \epsilon_0$, $1 - \Phi_N \ge \frac{3}{4} \epsilon_0$,  if $h$ is taken sufficiently small. 
% \begin{align}  
% & \frac{\Phi_N^{n+1} - \Phi_N^n}{\dt} = \Delta_h \mathcal{V}^{n+1} + \tau^{n+1}_* , \label{truncation-1} \\
% & \mathcal{V}^{n+1} = \ln ( 1+ \Phi_N^{n+1} ) - \ln (1 - \Phi_N^{n+1} ) - \theta_0 \Phi_N^n - \varepsilon^2 \Delta_h \Phi_N^{n+1}, \label{truncation-2} \\
% & \tilde{D}_{y} \mathcal{V}^{n+1}_{i, 0} = \tilde{D}_{y} \mathcal{V}^{n+1}_{i, N} = 0  ,  \quad 
%   (\Phi^{n+1}_N)_{i,0} = (\Phi^{B,n+1}_N)_{i} , \, \, \, (\Phi^{n+1}_N)_{i,N} = (\Phi^{T,n+1}_N)_{i} , \label{truncation-3}	\\
% & \frac{\Phi^{B,n+1}_N - \Phi^{B,n}_N}{\dt} = D_x^2 \mathcal{V}^{n+1}_B + \tau^{n+1}_{B,*}, \, \, \, 
%   \frac{\Phi^{T,n+1}_N - \Phi^{T,n}_N}{\dt} = D_x^2 \mathcal{V}^{n+1}_T + \tau^{n+1}_{T,*} , \label{truncation-4}	\\
% & \mathcal{V}^{n+1}_B = \ln ( 1+ \Phi^{B,n+1}_N ) - \ln (1 - \Phi^{B,n+1}_N ) - \theta_0 \Phi^{B,n}_N - \kappa\varepsilon D_x^2 \Phi^{B,n+1}_N - \varepsilon^2 (\tilde{D}_{y} \Phi_N^{n+1})_{\cdot, 0} ,\label{truncation-5}\\
% & \mathcal{V}^{n+1}_T = \ln ( 1+ \Phi^{T,n+1}_N ) - \ln (1 - \Phi^{T,n+1}_N ) - \theta_0 \Phi^{T,n}_N - \kappa\varepsilon D_x^2 \Phi^{T,n+1}_N + \varepsilon^2 (\tilde{D}_{y} \Phi_N^{n+1})_{\cdot, N} , \label{truncation-6}
% \end{align}
% where
% \begin{equation}
% \|\tau^{n+1}_*\|_2,\ \|\tau^{n+1}_{B,*}\|_{2,\Gamma},\ \|\tau^{n+1}_{T,*}\|_{2,\Gamma} \leq O(\dt+h^2).
% \end{equation}
Consequently, the discrete mass conservation of the projection solution on the boundary is valid:
\begin{align}
\overline{\Phi_N^{B,n+1}}^\Gamma &= \frac{1}{|\Gamma|}\int_\Gamma \Phi_N^{B,n+1} \dS = \frac{1}{|\Gamma|}\int_\Gamma \Phi_N^{B,n} \dS = \overline{\Phi_N^{B,n}}^\Gamma, \label{surface mass conservation 1} \\
\overline{\Phi_N^{T,n+1}}^\Gamma &= \frac{1}{|\Gamma|}\int_\Gamma \Phi_N^{T,n+1} \dS = \frac{1}{|\Gamma|}\int_\Gamma \Phi_N^{T,n} \dS = \overline{\Phi_N^{T,n}}^\Gamma, \label{surface mass conservation 2}
\end{align}
for any $n\in \mathbb{N}$. However, the Fourier projection in the $y$-direction is not available for the exact solution, which comes from the dynamical boundary condition. As an alternative, an auxiliary cosine function is introduced to enforce the bulk mass conservation at a discrete level:
\begin{equation}
\delta \Phi(x,y,t) = \left(\overline{\Phi_N}-\overline{\Phi_N^0}\right)\left( 1 - \cos(2 \pi y) \right), 
  \quad x, y \in [0,1] , \label{sine auxiliary function 1}
\end{equation}
where $\overline{\phi}=\frac{1}{|\Omega|}(\phi,1)$. This auxiliary function is a constant in the $x$-direction, and its independence on $x$ leads to the periodic nature (in the $x$-direction) and consistency with equation \eqref{DBC-Periodical-1}-\eqref{DBC-Periodical-7}. On the physical boundary sections, namely $y=0$ and $y=1$, this auxiliary function vanishes and satisfies the homogeneous Neumann boundary condition $\partial_n \delta \Phi |_{y=0,1} = 0$, so that it does not affect the discrete mass conservation on the boundary and the long stencil centered difference of the normal derivative. In addition, a trigonometric function is sufficiently smooth and regular. Under the reasonable assumption $\overline{\Phi_N^0}=\overline{\phi^0}$, the auxiliary function could be rewritten as
\begin{equation}
\delta \Phi(x,y,t) = (\overline{\Phi_N}-\overline{\Phi_N^0}) ( 1 - \cos(2 \pi y) ) = (\overline{\Phi_N}-\overline{\phi^0} ) ( 1 - \cos(2 \pi y) ). \label{sine auxiliary function 2}
\end{equation}
Such a cosine function will be useful in the subsequent analysis. 

The following lemma describes the core nature of the auxiliary function, based on the spatial discretization. The spatial truncation error bound could be obtained by using a straightforward Taylor expansion, as well as estimate \eqref{Fourier-approximation} for the projection solution:
\begin{align}
    & \partial_t \Phi_N = \Delta_h \mathcal{V} + \tau_h, \label{spatial semi-discretization 1} \\
    & \mathcal{V} = \ln ( 1+ \Phi_N ) - \ln (1 - \Phi_N ) - \theta_0 \Phi_N - \varepsilon^2 \Delta_h \Phi_N, \label{spatial semi-discretization 2} \\
    & \tilde{D}_{y} \mathcal{V}_{i, 0} = \tilde{D}_{y} \mathcal{V}_{i, N} = 0  ,  \quad 
  (\Phi_N)_{i,0} = (\Phi^{B}_N)_{i} , \, \, \, (\Phi_N)_{i,N} = (\Phi^{T}_N)_{i}, \label{spatial semi-discretization 3} \\
    & \partial_t \Phi^{B}_N = D_x^2 \mathcal{V}_B + \tau_{B,h}, \, \, \,  \partial_t \Phi^{T}_N = D_x^2 \mathcal{V}_T + \tau_{T,h} , \label{spatial semi-discretization 4}	\\
    & \mathcal{V}_B = \ln ( 1+ \Phi^{B}_N ) - \ln (1 - \Phi^{B}_N ) - \theta_0 \Phi^{B}_N - \kappa\varepsilon D_x^2 \Phi^{B}_N - \varepsilon^2 (\tilde{D}_{y} \Phi_N)_{\cdot, 0} ,\label{spatial semi-discretization 5}\\
& \mathcal{V}_T = \ln ( 1+ \Phi^{T}_N ) - \ln (1 - \Phi^{T}_N ) - \theta_0 \Phi^{T}_N - \kappa\varepsilon D_x^2 \Phi^{T}_N + \varepsilon^2 (\tilde{D}_{y} \Phi_N)_{\cdot, N} , \label{spatial semi-discretization 6}
\end{align}
where
\begin{equation}
\|\tau_h\|_2,\ \|\tau_{B,h}\|_{2,\Gamma},\ \|\tau_{T,h}\|_{2,\Gamma} \leq O(h^2).
\end{equation}

\begin{lem} \label{sine function property}
    The auxiliary cosine function $\delta \Phi$, defined by \eqref{sine auxiliary function 1}, satisfies the following properties.
    \begin{itemize}
        \item[$(1)$] $\delta \Phi$ is infinitesimal under the discrete $\ell^2$-norm and has order
        \begin{equation}
            \|\delta \Phi^n\|_2 \leq C h^2,\quad 1 \leq n \leq N ,  \label{auxiliary function order}
        \end{equation}
        where $\delta \Phi^n=\delta \Phi(\cdot,\cdot,t_n)$ and $C$ is a constant independent on $\dt$ and $h$.
        \item[$(2)$] The following corrected discrete mass conservation is valid: 
        \begin{equation}
            \overline{\Phi_N-\delta \Phi}=\overline{\Phi_N^0}=\overline{\phi^0}.  \label{bulk mass conservation}
        \end{equation}
    \end{itemize}
\end{lem}
\begin{proof}
    It is noticed that the cosine function has enough regularity and the infinitesimal property of $\delta \Phi$ is only dependent on the discrete mass change $(\overline{\Phi_N}-\overline{\Phi_N^0})$. Performing a discrete summation on both sides of \eqref{spatial semi-discretization 1}, with the help of the summation-by-part formula under homogeneous Neumann boundary condition, we observe the following equality:
    \begin{equation}
        \partial_t  \overline{\Phi_N} = \overline{\tau_h} .\label{mass order 1}
    \end{equation}
    In turn, an integration in time from $t_0$ to $t_n$ gives 
    \begin{align}
        \overline{\Phi^{n}_N}-\overline{\Phi_N^0} = \int_{0}^{t_n} \overline{\tau_h} dt \leq \int_{0}^{T} |\overline{\tau_h}| dt
         \leq C_{\Omega, T} h^2, \label{mass order 2}
    \end{align}
    where the coefficient $C_{\Omega, T}$ is only dependent on $\Omega$ and final time $T$, independent of $\dt$ and $h$. Subsequently, \eqref{auxiliary function order} becomes a straightforward consequence of  \eqref{mass order 2}.

    Identity \eqref{bulk mass conservation} comes from a direct calculation. Based on \eqref{sine auxiliary function 1} and the definition of the discrete $\ell^2$ inner product, we see that 
    \begin{align}
        \overline{\Phi_N-\delta \Phi} & = \overline{\Phi_N}-\overline{\delta \Phi} = \overline{\Phi_N} - (\overline{\Phi_N}-\overline{\Phi_N^0}) h \sum_{j=1}^N \left( 1- \cos(2 \pi jh) \right) \nonumber \\
	& = \overline{\Phi_N} - (\overline{\Phi_N}-\overline{\Phi_N^0})\left(1 - \frac{{\rm e}^{2\pi ih}}{2} \frac{h(1- {\rm e}^{2 \pi i})}{1-{\rm e}^{2 \pi ih}}- \frac{{\rm e}^{-2\pi ih}}{2} \frac{h(1- {\rm e}^{-2 \pi i})}{1- {\rm e}^{-2 \pi ih}}\right)  \nonumber \\
	& = \overline{\Phi_N} - (\overline{\Phi_N}-\overline{\Phi_N^0}) 
	 = \overline{\Phi_N^0} ,  \label{bulk mass conservation proof}
    \end{align}
in which the classic summation of cosine functions is applied:
    \begin{equation}
        \sum_{j=1}^N \cos(2jx) =\frac{{\rm e}^{i2x}}{2} \frac{1- {\rm e}^{i2Nx}}{1- {\rm e}^{i2x}} + \frac{{\rm e}^{-i2x}}{2} \frac{1- {\rm e}^{-i2Nx}}{1-{\rm e}^{-i2x}} . \label{summation of sine}
    \end{equation}
\end{proof}
%In fact, based on \eqref{mass order 1}, the auxiliary sine function defined as \eqref{sine auxiliary function 1} has the same order with the truncation error, which means the convergence analysis of the higher-order scheme is also valid through the same approach. 

As a result of Lemma \ref{sine function property}, we proceed into an introduction of a corrected solution 
\begin{equation}
\hat{\Phi}_N=\Phi_N-\delta\Phi . \label{corrected solution}
\end{equation}
With a temporal discretization, a modified error function is defined as 
\begin{equation}
\tilde{\phi}^n=\hat{\Phi}_N^n-\phi^n, \quad n \in \mathbb{N}. \label{define error function}
\end{equation}
On one hand, \eqref{bulk mass conservation} ensures the bulk mass conservation of $\hat{\Phi}_N$ as the surface mass conservation \eqref{surface mass conservation 1}-\eqref{surface mass conservation 2} is not affected. The homogeneous Neumann boundary condition of the cosine function plays an important role in this property. On the other hand, \eqref{auxiliary function order} indicates that the auxiliary function has the same spatial order as the truncation error and is independent with time step, so that it has no effect on the convergence analysis. In turn, a substitution of the  corrected profile gives the following fully discrete truncation error estimate: 
\begin{align}  
& \frac{\hat{\Phi}_N^{n+1} - \hat{\Phi}_N^n}{\dt} = \Delta_h \hat{\mathcal{V}}^{n+1} + \tau^{n+1} , \label{truncation-2.1} \\
& \hat{\mathcal{V}}^{n+1} = \ln ( 1+ \hat{\Phi}_N^{n+1} ) - \ln (1 - \hat{\Phi}_N^{n+1} ) - \theta_0 \hat{\Phi}_N^n - \varepsilon^2 \Delta_h \hat{\Phi}_N^{n+1}, \label{truncation-2.2} \\
& \tilde{D}_{y} \hat{\mathcal{V}}^{n+1}_{i, 0} = \tilde{D}_{y} \hat{\mathcal{V}}^{n+1}_{i, N} = 0  ,  \quad 
  (\hat{\Phi}^{n+1}_N)_{i,0} = (\hat{\Phi}^{B,n+1}_N)_{i} , \, \, \, (\hat{\Phi}^{n+1}_N)_{i,N} = (\hat{\Phi}^{T,n+1}_N)_{i} , \label{truncation-2.3}	\\
& \frac{\hat{\Phi}_N^{B,n+1} - \hat{\Phi}_N^{B,n}}{\dt} = D_x^2 \hat{\mathcal{V}}^{n+1}_B + \tau^{n+1}_B, \, \, \, 
  \frac{\hat{\Phi}_N^{T,n+1} - \hat{\Phi}_N^{T,n}}{\dt} = D_x^2 \hat{\mathcal{V}}^{n+1}_T + \tau^{n+1}_T , \label{truncation-2.4}	\\
& \hat{\mathcal{V}}^{n+1}_B = \ln ( 1+ \hat{\Phi}^{B,n+1}_N ) - \ln (1 - \hat{\Phi}^{B,n+1}_N ) - \theta_0 \hat{\Phi}^{B,n}_N - \varepsilon\kappa D_x^2 \hat{\Phi}^{B,n+1}_N - \varepsilon^2 (\tilde{D}_{y} \hat{\Phi}_N^{n+1})_{\cdot, 0} ,\label{truncation-2.5}\\
& \hat{\mathcal{V}}^{n+1}_T = \ln ( 1+ \hat{\Phi}^{T,n+1}_N ) - \ln (1 - \hat{\Phi}^{T,n+1}_N ) - \theta_0 \hat{\Phi}^{T,n}_N - \varepsilon\kappa D_x^2 \hat{\Phi}^{T,n+1}_N + \varepsilon^2 (\tilde{D}_{y} \hat{\Phi}_N^{n+1})_{\cdot, N} \label{truncation-2.6}
\end{align}
where 
\begin{equation}
\|\tau^{n+1}\|_2,\ \|\tau^{n+1}_B\|_{2,\Gamma},\ \|\tau^{n+1}_T\|_{2,\Gamma} \leq O(\dt+h^2). \label{truncation-2.7}
\end{equation}
The separation property, $1 + \hat{\Phi}_N \ge \frac{1}{2} \epsilon_0$, $1 - \hat{\Phi}_N \ge \frac{1}{2} \epsilon_0$, also holds provided that $\dt$ and $h$ are sufficiently small. Based on the separation property of the corrected solution, the consistency analysis of the logarithmic terms in the chemical potential expansion \eqref{truncation-2.2} comes from a standard Taylor expansion
\begin{align}
& \ln (1+\hat{\Phi}_N^{n+1})=\ln (1+\Phi_N^{n+1}-\delta\Phi^{n+1}) = \ln(1+\Phi_N^{n+1})-\frac{\delta\Phi^{n+1}}{1+\Phi_{1,*}}, \\
& \ln (1-\hat{\Phi}_N^{n+1})=\ln (1-\Phi_N^{n+1}+\delta\Phi^{n+1}) = \ln(1-\Phi_N^{n+1})+\frac{\delta\Phi^{n+1}}{1-\Phi_{2,*}} , 
\end{align}
where both $\Phi_{1,*}$ and $\Phi_{2,*}$ are located between $\Phi_N^{n+1}$ and $\hat{\Phi}_N^{n+1}$, at a  point-wise level. The separation inequalities, $1\pm\Phi_{1,*} \ge \frac{1}{2} \epsilon_0$, $1\pm\Phi_{2,*} \ge \frac{1}{2} \epsilon_0$, are also valid. A combination of the separability and \eqref{auxiliary function order} (in Lemma \ref{sine function property}) leads to a full order estimate
\begin{equation}
\| \frac{\delta\Phi^{n+1}}{1+\Phi_{1,*}}\|_2,\ \| \frac{\delta\Phi^{n+1}}{1-\Phi_{1,*}}\|_2 \leq \frac{C}{\epsilon_0}(\dt+h^2) ,  \label{order estimate}
\end{equation}
which has already been involved in the truncation error \eqref{truncation-2.7}. In turn, subtracting the numerical solution \eqref{scheme-CHDBC-1}-\eqref{scheme-CHDBC-5} from \eqref{truncation-2.1}-\eqref{truncation-2.6} gives
\begin{align}
& \frac{\tilde{\phi}^{n+1}-\tilde{\phi}^n}{\dt} = \Delta_h \tilde{\mu}^{n+1} + \tau^{n+1} , \label{error eq-1} \\
& \tilde{\mu}^{n+1} = \ln(1+\hat{\Phi}_N^{n+1})-\ln(1-\hat{\Phi}_N^{n+1})-\ln(1+\phi^{n+1})+\ln(1-\phi^{n+1}) - \theta_0 \tilde{\phi}^n - \varepsilon^2 \Delta_h \tilde{\phi}^{n+1}, \label{error eq-2} \\
& \tilde{D}_{y} \tilde{\mu}^{n+1}_{i, 0} = \tilde{D}_{y} \tilde{\mu}^{n+1}_{i, N} = 0  ,  \quad 
  \tilde{\phi}^{n+1}_{i,0} = \tilde{\phi}^{B,n+1}_{i} , \, \, \, \tilde{\phi}^{n+1}_{i,N} =\tilde{\phi}^{T,n+1}_{i} , \label{error eq-3}	\\
& \frac{\tilde{\phi}^{B,n+1} - \tilde{\phi}^{B,n}}{\dt} = D_x^2 \tilde{\mu}^{n+1}_B + \tau^{n+1}_B, \, \, \, 
  \frac{\tilde{\phi}^{T,n+1} - \tilde{\phi}^{T,n}}{\dt} = D_x^2 \tilde{\mu}^{n+1}_T + \tau^{n+1}_T, \label{error eq-4}	\\
& \tilde{\mu}^{n+1}_B = \ln(1+\hat{\Phi}_N^{B,n+1})-\ln(1-\hat{\Phi}_N^{B,n+1})-\ln(1+\phi^{B,n+1})+\ln(1-\phi^{B,n+1}) - \theta_0 \tilde{\phi}^{B,n} \nonumber \\
& \quad \quad \quad - \varepsilon\kappa D_x^2 \tilde{\phi}^{B,n+1} - \varepsilon^2 \tilde{D}_{y} \tilde{\phi}^{n+1}_{\cdot, 0} ,\label{error eq-5}\\
& \tilde{\mu}^{n+1}_T = \ln(1+\hat{\Phi}_N^{T,n+1})-\ln(1-\hat{\Phi}_N^{T,n+1})-\ln(1+\phi^{T,n+1})+\ln(1-\phi^{T,n+1}) - \theta_0 \tilde{\phi}^{T,n} \nonumber \\
& \quad \quad \quad - \varepsilon\kappa D_x^2 \tilde{\phi}^{T,n+1} + \varepsilon^2 \tilde{D}_{y} \tilde{\phi}^{n+1}_{\cdot, N}. \label{error eq-6}
\end{align}
Subsequently, by the above correction approach, it follows that $\overline{\tilde{\phi}^{B,n}}^\Gamma = \overline{\tilde{\phi}^{T,n}}^\Gamma = \overline{\tilde{\phi}^n}=0$, for any $n \in \mathbb{N}$, so that both $\| \cdot \|_{-1,\Gamma}$ and $\| \cdot \|_{-1}$ are well defined for the numerical error grid function $\tilde{\phi}^n$ and the boundary projection. In addition, the following summation-by-part formula is valid for any $\psi \in \Cperx$, because of the special boundary property of $\delta\Phi$: 
\begin{align}
    ( \Delta_h \tilde{\phi}^{n+1},\psi ) &= - ( \nabla_h \tilde{\phi}^{n+1}, \nabla_h \psi  ) + (\tilde{D}_y \tilde{\phi}^{n+1}_{\cdot , N}, \psi_{\cdot , N})_\Gamma - (\tilde{D}_y \tilde{\phi}^{n+1}_{\cdot, 0}, \psi_{\cdot, 0})_\Gamma  \nonumber  \\
    & = - ( \nabla_h \tilde{\phi}^{n+1}, \nabla_h \psi  ) + (\tilde{D}_y (\hat{\Phi}_N)^{n+1}_{\cdot , N}, \psi_{\cdot , N})_\Gamma - (\tilde{D}_y (\hat{\Phi}_N)^{n+1}_{\cdot, 0}, \psi_{\cdot, 0})_\Gamma . 
\end{align}

\begin{rem}
Such an auxiliary function approach has also been reported in the higher-order consistency analysis of the singular chemical potential \cite{Dong2022a, GuoY24, LiuC2021a}, where a strong correction function is based on the Taylor expansion, and a higher order truncation error is reached via an asymptotic analysis. In the convergence analysis of the dynamical boundary condition, the auxiliary function turns out to be a weak version, in the sense that the truncation error still holds the initial rate. Of course, it could be combined with higher-order consistency analysis or other theoretical approaches.
\end{rem}

\begin{rem}\label{order remark}
Two requirements play an important role in the choice of auxiliary function. The first is a proper boundary condition, in comparison with the physical boundary in the original PDE. The mass conservation on the physical surface and the summation-by-part formula must not be violated. This is the reason why we choose the trigonometric function that simultaneously satisfies both the homogeneous Neumann boundary condition and homogeneous Dirichlet boundary condition. The second is a sufficient regularity required by the Taylor expansion of the logarithmic terms in \eqref{truncation-2.1}-\eqref{truncation-2.6}. In terms of the accuracy order, it is independent of temporal discretization and can always be covered by the spatial truncation error, which implies that this approach is also valid for the higher-order schemes, as well as for other boundary conditions. 
\end{rem}

\subsection{The convergence estimate in the $H^{-1}$ norm}

The following theorem is the main result of this article.
\begin{theorem}\label{Theorem 1}
Given smooth initial data $\Phi(\cdot,t=0)$, suppose the exact solution for CH equation with dynamical boundary condition has regularity of class $\mathcal{R}$. Provided that $\dt$ and $h$ are sufficiently small, we have
\begin{equation}
\|\tilde{\phi}^{n+1}\|_{-1}+\|\tilde{\phi}^{n+1}\|_{-1,\Gamma} + \dt\sum_{k=1}^n \left( \varepsilon^2 \|\nabla_h \tilde{\phi}^{k}\|_2^2 + \varepsilon\kappa \|D_x\tilde{\phi}^{k}\|_{2,\Gamma}^2 \right) \leq C_{\Omega,T}^*(\dt + h^2) , 
 \label{convergence theorem} 
\end{equation}
where for simplicity
\begin{align}
&\|\tilde{\phi}^{n+1}\|_{-1,\Gamma}^2:=\|\tilde{\phi}^{T,n+1}\|_{-1,\Gamma}^2 
 +\|\tilde{\phi}^{B,n+1}\|_{-1,\Gamma}^2, \label{simplicity -1 norm} \\
& \|D_x\tilde{\phi}^{k}\|_{2,\Gamma}^2 := \|D_x\tilde{\phi}^{k,T}\|_{2,\Gamma}^2 + \|D_x\tilde{\phi}^{k,B}\|_{2,\Gamma}^2, \quad 1 \leq k \leq n. \label{simplicity -1 norm 0}
\end{align}
The positive constant $C_{\Omega,T}^*$ is only dependent on the exact solution, as well as $\Omega$ and $T$.  
\end{theorem}

\begin{proof}
Taking an inner product with \eqref{error eq-1} by $(-\Delta_h)^{-1}\tilde{\phi}^{n+1}$ gives
\begin{align}
\frac{1}{2\dt}& ( \|\tilde{\phi}^{n+1}\|_{-1}^2-\|\tilde{\phi}^n\|_{-1}^2 )-\varepsilon^2 (\Delta_h \tilde{\phi}^{n+1} ,\tilde{\phi}^{n+1} ) \nonumber \\
& +\left(\ln(1+\hat{\Phi}_N^{n+1})-\ln(1+\phi^{n+1})-\ln(1-\hat{\Phi}_N^{n+1})+\ln(1-\phi^{n+1}),\tilde{\phi}^{n+1}\right) \nonumber \\
& \leq \theta_0 (\tilde{\phi}^n,\tilde{\phi}^{n+1} ) + (\tau^{n+1}, (-\Delta_h)^{-1}\tilde{\phi}^{n+1} ). \label{convergence 0.0}
\end{align}
Inequality \eqref{convergence 0.0} could be divided into three parts:
\begin{align}
I_1&=-\varepsilon^2 (\Delta_h \tilde{\phi}^{n+1} ,\tilde{\phi}^{n+1} ), \label{convergence 0.1} \\
I_2&=\left(\ln(1+\hat{\Phi}_N^{n+1})-\ln(1+\phi^{n+1})-\ln(1-\hat{\Phi}_N^{n+1})+\ln(1-\phi^{n+1}),\tilde{\phi}^{n+1}\right), \label{convergence 0.2} \\
I_3&=\theta_0 (\tilde{\phi}^n,\tilde{\phi}^{n+1} ) + (\tau^{n+1}, (-\Delta_h)^{-1}\tilde{\phi}^{n+1} ). \label{convergence 0.3}
\end{align}
An application of the Cauchy inequality implies that
\begin{align}
&\theta_0 ( \tilde{\phi}^n,\tilde{\phi}^{n+1} )\leq \theta_0\|\tilde{\phi}^n\|_{-1} 
\cdot \|\nabla_h\tilde{\phi}^{n+1}\|_2 \leq \frac{\varepsilon^2}{2}\|\nabla_h\tilde{\phi}^{n+1}\|_2^2 + \frac{\theta_0^2}{2\varepsilon^2}\|\tilde{\phi}^n\|_{-1}^2, \label{convergence 0.4} \\
& ( \tau^{n+1}, (-\Delta_h)^{-1}\tilde{\phi}^{n+1} ) \leq \|\tau^{n+1}\|_{-1} 
\cdot \|\tilde{\phi}^{n+1}\|_{-1} \leq \|\tau^{n+1}\|_{-1}^2 
+ \frac{1}{4}\| \tilde{\phi}^{n+1} \|_{-1}^2. \label{convergence 0.6}
\end{align}
Then we get  
\begin{align}
I_3 \leq \frac{\varepsilon^2}{2}\|\nabla_h\tilde{\phi}^{n+1}\|_2^2 + \frac{\theta_0^2}{2\varepsilon^2}\|\tilde{\phi}^n\|_{-1}^2 + \|\tau^{n+1}\|_{-1}^2 + \frac{1}{4}\| \tilde{\phi}^{n+1} \|_{-1}^2. \label{convergence 0.5}
\end{align}

Regarding the logarithmic terms, a combination of the convexity and the fact that $-1 < \hat{\Phi}_N, \phi^{n+1} < 1$ (at a point-wise level) leads to
\begin{align}
&\left(\ln(1+\hat{\Phi}_N^{n+1})-\ln(1+\phi^{n+1}),\tilde{\phi}^{n+1}\right) \geq 0, \label{convergence 1.1} \\
&-\left(\ln(1-\hat{\Phi}_N^{n+1})-\ln(1-\phi^{n+1}),\tilde{\phi}^{n+1}\right) \geq 0 \label{convergence 1.2} . 
\end{align}
This ensures that $I_2$ is non-negative.

For the diffusion term, the dynamical boundary condition makes the analysis more complicated: 
\begin{align}
I_1=\varepsilon^2 \|\nabla_h \tilde{\phi}^{n+1}\|_2^2-\varepsilon^2(\tilde{D}_{y}\tilde{\phi}^{n+1}_{\cdot,N},\tilde{\phi}^{n+1}_{\cdot,N})_\Gamma+\varepsilon^2(\tilde{D}_{y}\tilde{\phi}^{n+1}_{\cdot,0},\tilde{\phi}^{n+1}_{\cdot,0})_\Gamma. \label{convergence 2.1}
\end{align}
Taking an inner product with the bottom boundary condition \eqref{error eq-4} by $(-D_x^2)\tilde{\phi}^{B,n+1}$ yields 
\begin{align}
    \frac{1}{2\dt}& ( \|\tilde{\phi}^{B,n+1}\|_{-1,\Gamma}^2-\|\tilde{\phi}^{B,n}\|_{-1,\Gamma}^2 )+\varepsilon\kappa\| D_x\tilde{\phi}^{B,n+1} \|_{2,\Gamma}^2 \nonumber \\
& +\left(\ln(1+\hat{\Phi}_N^{B,n+1})-\ln(1+\phi^{B,n+1})-\ln(1-\hat{\Phi}_N^{B,n+1})+\ln(1-\phi^{B,n+1}),\tilde{\phi}^{B,n+1}\right)_\Gamma \nonumber \\
& \leq \theta_0 (\tilde{\phi}^{B,n},\tilde{\phi}^{B,n+1} )_\Gamma + (\tau_B^{n+1}, (-D_x^2)^{-1}\tilde{\phi}^{B,n+1} )_\Gamma + \varepsilon^2 (\tilde{D}_{y}\tilde{\phi}^{n+1}_{\cdot,0},\tilde{\phi}^{B,n+1} )_\Gamma, \label{convergence 2.2}
\end{align}
in which the one-dimensional the summation-by-part formula is applied, with periodic boundary condition.  Again, utilizing the positivity of the convex logarithmic inner product term and the standard H{\"o}lder inequality, we have
\begin{align}
    \frac{1}{2\dt}& ( \|\tilde{\phi}^{B,n+1}\|_{-1,\Gamma}^2-\|\tilde{\phi}^{B,n}\|_{-1,\Gamma}^2 )+\varepsilon\kappa\| D_x\tilde{\phi}^{B,n+1} \|_{2,\Gamma}^2 \nonumber \\
    & \leq \theta_0(\tilde{\phi}^{B,n},\tilde{\phi}^{B,n+1})_\Gamma+(\tau_B^{n+1}, (-D_x^2)^{-1}\tilde{\phi}^{B,n+1})_\Gamma + \varepsilon^2(\tilde{D}_{y}\tilde{\phi}^{n+1}_{\cdot,0},\tilde{\phi}^{B,n+1})_\Gamma  \nonumber \\
    & \leq \frac{\varepsilon\kappa}{2}\|D_x\tilde{\phi}^{B,n+1}\|_{2,\Gamma}^2+\frac{\theta_0^2}{2\varepsilon\kappa} \|\tilde{\phi}^{B,n} \|_{-1,\Gamma}^2 + \|\tau_B^{n+1}\|_{-1,\Gamma}^2   \nonumber \\
    &  \qquad\qquad\qquad\qquad\qquad + \frac{1}{4} \|\tilde{\phi}^{B,n+1}\|_{-1,\Gamma}^2 + \varepsilon^2(\tilde{D}_{y}\tilde{\phi}^{n+1}_{\cdot,0},\tilde{\phi}^{B,n+1})_\Gamma. \label{convergence 2.3}
\end{align}
Furthermore, we see that 
\begin{align}
    \frac{1}{2\dt}& ( \|\tilde{\phi}^{B,n+1}\|_{-1,\Gamma}^2-\|\tilde{\phi}^{B,n}\|_{-1,\Gamma}^2 )+\frac{\varepsilon\kappa}{2}\|D_x\tilde{\phi}^{B,n+1}\|_{2,\Gamma}^2 \nonumber \\
    & \leq \frac{\theta_0^2}{2\varepsilon\kappa} \|\tilde{\phi}^{B,n} \|_{-1,\Gamma}^2 + \frac{1}{4} \|\tilde{\phi}^{B,n+1}\|_{-1,\Gamma}^2 + \|\tau_B^{n+1}\|_{-1,\Gamma}^2 +\varepsilon^2(\tilde{D}_{y}\tilde{\phi}^{n+1}_{\cdot,0},\tilde{\phi}^{n+1}_{\cdot,0})_\Gamma. \label{convergence 2.4}
\end{align}
A similar inequality could be derived on the top boundary section: 
\begin{align}
    \frac{1}{2\dt}& ( \|\tilde{\phi}^{T,n+1}\|_{-1,\Gamma}^2-\|\tilde{\phi}^{T,n}\|_{-1,\Gamma}^2 ) +\frac{\varepsilon\kappa}{2}\|D_x\tilde{\phi}^{T,n+1}\|_{2,\Gamma}^2 \nonumber \\
    & \leq \frac{\theta_0^2}{2\varepsilon\kappa} \|\tilde{\phi}^{T,n} \|_{-1,\Gamma}^2  + \frac{1}{4} \|\tilde{\phi}^{T,n+1}\|_{-1,\Gamma}^2+ \|\tau_T^{n+1}\|_{-1,\Gamma}^2 -\varepsilon^2(\tilde{D}_{y}\tilde{\phi}^{n+1}_{\cdot,N},\tilde{\phi}^{n+1}_{\cdot,N})_\Gamma. \label{convergence 2.5}
\end{align}
Therefore, a combination of \eqref{convergence 2.1}, \eqref{convergence 2.4} and \eqref{convergence 2.5} results in 
\begin{align}
I_1 \geq & \frac{1}{2\dt}(\|\tilde{\phi}^{n+1}\|^2_{-1,\Gamma}-\|\tilde{\phi}^n\|^2_{-1,\Gamma})+\varepsilon^2\|\nabla_h \tilde{\phi}^{n+1}\|_2^2+\frac{\varepsilon\kappa}{2} \|D_x\tilde{\phi}^{n+1}\|_{2,\Gamma}^2 \nonumber \\
& -\frac{\theta_0^2}{2\varepsilon\kappa}\|\tilde{\phi}^n\|_{-1,\Gamma}^2-\frac{1}{4}\|\tilde{\phi}^{n+1}\|_{-1,\Gamma}^2-\|\tau^{n+1}_T\|_{-1,\Gamma}^2-\|\tau^{n+1}_B\|_{-1,\Gamma}^2 , 
\label{convergence 2.7}
\end{align}
where for simplicity, we have denoted 
\begin{align}
&\|\phi\|_{-1,\Gamma}^2:=\|\phi^T\|_{-1,\Gamma}^2+\|\phi^B\|_{-1,\Gamma}^2, \quad \forall \phi \in {\mathcal C}_{\rm per}^x \ \text{satisfying}\  \overline{\phi^T}^\Gamma=\overline{\phi^B}^\Gamma=0 ,  
  \label{convergence explain} \\
&\|D_x \phi\|_{2,\Gamma}^2:=\|D_x \phi^{B}\|_{2,\Gamma}^2 + \|D_x \phi^{T}\|_{2,\Gamma}^2, \quad \forall \phi \in {\mathcal C}_{\rm per}^x. 
\label{convergence explain 0}
\end{align}
Next, a summation of estimates for $I_1$, $I_2$ and $I_3$ gives 
\begin{align}
\frac{1}{2\dt}&\left(\|\tilde{\phi}^{n+1}\|_{-1}^2+\|\tilde{\phi}^{n+1}\|^2_{-1,\Gamma}-\|\tilde{\phi}^n\|_{-1}^2-\|\tilde{\phi}^n\|^2_{-1,\Gamma}\right)+\frac{\varepsilon^2}{2}\|\nabla_h \tilde{\phi}^{n+1}\|_2^2 + \frac{\varepsilon\kappa}{2} \|D_x\tilde{\phi}^{n+1}\|_{2,\Gamma}^2  \nonumber \\
& \leq \frac{\theta_0^2}{2\varepsilon^2}\|\tilde{\phi}^n\|_{-1}^2 + \frac{1}{4}\| \tilde{\phi}^{n+1} \|_{-1}^2+\frac{\theta_0^2}{2\varepsilon\kappa}\|\tilde{\phi}^n\|_{-1,\Gamma}^2+\frac{1}{4}\|\tilde{\phi}^{n+1}\|_{-1,\Gamma}^2 \nonumber \\
& \quad \quad + \|\tau^{n+1}\|_{-1}^2 +\|\tau^{n+1}_T\|_{-1,\Gamma}^2
 +\|\tau^{n+1}_B\|_{-1,\Gamma}^2 .  \label{convergence 3.1}
\end{align}
A further summation in time and reorganization reveals that 
\begin{align}
    \|\tilde{\phi}^{n+1}\|_{-1}^2 & +\|\tilde{\phi}^{n+1}\|^2_{-1,\Gamma} + \dt\sum_{k=1}^n \left( \varepsilon^2 \|\nabla_h \tilde{\phi}^{k}\|_2^2 + \varepsilon\kappa \|D_x\tilde{\phi}^{k}\|_{2,\Gamma}^2 \right) \nonumber \\
    &\leq 2\dt C_\alpha \sum_{k=1}^n \left( \|\tilde{\phi}^n\|_{-1}^2+\|\tilde{\phi}^n\|^2_{-1,\Gamma}\right) + \frac{\dt}{2} \left(\|\tilde{\phi}^{n+1}\|_{-1}^2+\|\tilde{\phi}^{n+1}\|^2_{-1,\Gamma} \right) \nonumber \\
    & \qquad\qquad + 2\dt \sum_{k=1}^{n} \left(\|\tau^k\|_{-1}^2+ \|\tau^k_T\|_{-1,\Gamma}^2 + \|\tau^k_B\|_{-1,\Gamma}^2\right) \nonumber \\
    & \leq 2\dt C_\alpha \sum_{k=1}^n \left( \|\tilde{\phi}^n\|_{-1}^2+\|\tilde{\phi}^n\|^2_{-1,\Gamma}\right) + \frac{1}{2} \left(\|\tilde{\phi}^{n+1}\|_{-1}^2+\|\tilde{\phi}^{n+1}\|^2_{-1,\Gamma} \right) \nonumber \\
    & \qquad\qquad + 2\dt \sum_{k=1}^{n} \left(\|\tau^k\|_{-1}^2+ \|\tau^k_T\|_{-1,\Gamma}^2 + \|\tau^k_B\|_{-1,\Gamma}^2\right) , \label{convergence 3.2}
\end{align}
and 
\begin{align}
    \|\tilde{\phi}^{n+1}\|_{-1}^2 & +\|\tilde{\phi}^{n+1}\|^2_{-1,\Gamma} +  2\dt\sum_{k=1}^n \left( \varepsilon^2 \|\nabla_h \tilde{\phi}^{k}\|_2^2 + \varepsilon\kappa \|D_x\tilde{\phi}^{k}\|_{2,\Gamma}^2 \right) \nonumber \\
   & \leq  4\dt C_\alpha \sum_{k=1}^n \Big( \|\tilde{\phi}^n\|_{-1}^2+\|\tilde{\phi}^n\|^2_{-1,\Gamma} \Big) + 4\dt \sum_{k=1}^{n} \Big(\|\tau^k\|_{-1}^2 +\|\tau^k_T\|_{-1,\Gamma}^2+\|\tau^k_B\|_{-1,\Gamma}^2 \Big) , \label{convergence 3.3}
\end{align}
where $C_\alpha = \max \{ \frac{\theta_0^2}{2\varepsilon^2}, \frac{\theta_0^2}{2\varepsilon\kappa} \}$ and the assumption $\dt < 1$ has been used. Meanwhile, the truncation error estimate is based on the fact that $\| \cdot \|_{-1} \leq C\| \cdot \|_2$: 
\begin{align}
&\|\tau^{n+1}\|_{-1} \leq C\|\tau^{n+1}\|_2 \leq C(\dt+h^2), \label{convergence 4.1} \\
&\|\tau^{n+1}_T\|_{-1,\Gamma} \leq C\|\tau^{n+1}_T\|_{2,\Gamma} \leq C(\dt+h^2), \label{convergence 4.2} \\
&\|\tau^{n+1}_B\|_{-1,\Gamma} \leq C\|\tau^{n+1}_B\|_{2,\Gamma} \leq C(\dt+h^2). \label{convergence 4.3} 
\end{align}
Therefore, with sufficiently small $\dt$ and $h$, an application of discrete Gronwall inequality leads to the desired convergence estimate
\begin{equation}
\|\tilde{\phi}^{n+1}\|_{-1}+\|\tilde{\phi}^{n+1}\|_{-1,\Gamma} + \dt\sum_{k=1}^n \left( \varepsilon^2 \|\nabla_h \tilde{\phi}^{k}\|_2^2 + \varepsilon\kappa \|D_x\tilde{\phi}^{k}\|_{2,\Gamma}^2 \right) \leq C_{\Omega,T}^*(\dt + h^2). \label{convergence 3.5} 
\end{equation}
This completes the proof.
\end{proof}

\section{The second order numerical scheme} \label{sec: 2nd numerical scheme}	
As mentioned in the previous section, since the auxiliary function has the same accuracy order as the spatial truncation error, the idea of this convergence analysis could be be extended to higher-order (in time) numerical schemes. In this section, we will propose a second order numerical scheme, based on a modified BDF2 temporal discretization. In this numerical approach, artificial stabilizers for both bulk and boundary are needed, and a modified energy dissipation could be derived. The numerical scheme is proposed as: given $\phi^{n}$, $\phi^{n-1} \in \Cperx$, find $\phi^{n+1}$, $\mu^{n+1} \in \Cperx$ such that
\begin{align}  
& \frac{3\phi^{n+1} - 4\phi^n + \phi^{n-1}}{2\dt} = \Delta_h \mu^{n+1}, \label{2nd-scheme-CHDBC-1}\\
& \mu^{n+1} = \ln ( 1+ \phi^{n+1} ) - \ln (1 - \phi^{n+1} ) - \theta_0 \hat{\phi}^{n+1}  - \varepsilon^2 \Delta_h \phi^{n+1} - A \dt \Delta_h (\phi^{n+1}-\phi^n), \label{2nd-scheme-CHDBC-2} \\
& \tilde{D}_{y} \mu^{n+1}_{i, 0} = \tilde{D}_{y} \mu^{n+1}_{i, N} = 0  ,  \quad 
  \phi^{n+1}_{i,0} = \phi^{B,n+1}_{i} , \, \, \, \phi^{n+1}_{i,N} = \phi^{T,n+1}_{i} , \label{2nd-scheme-CHDBC-3}	\\
& \frac{3\phi^{B,n+1} - 4\phi^{B,n} + \phi^{B,n-1}}{2\dt} = D_x^2 \mu^{n+1}_B, \, \, \, 
  \frac{3\phi^{T,n+1} - 4\phi^{T,n} + \phi^{T,n-1}}{2\dt} = D_x^2 \mu^{n+1}_T , \label{2nd-scheme-CHDBC-4}	\\
& \mu^{n+1}_B = \ln ( 1+ \phi^{B,n+1} ) - \ln (1 - \phi^{B,n+1} ) - \theta_0 \hat{\phi}^{B,n+1} - \varepsilon\kappa D_x^2 \phi^{B,n+1} - \varepsilon^2 \tilde{D}_{y} \phi^{n+1}_{\cdot, 0} \nonumber \\
& \quad\quad\quad - A \dt \tilde{D}_{y}(\phi^{n+1}-\phi^n)_{\cdot,0}-B \dt D_x^2(\phi^{B,n+1}-\phi^{B,n}),\label{2nd-scheme-CHDBC-5}\\
& \mu^{n+1}_T = \ln ( 1+ \phi^{T,n+1} ) - \ln (1 - \phi^{T,n+1} ) - \theta_0 \hat{\phi}^{T,n+1} - \varepsilon\kappa D_x^2 \phi^{T,n+1} + \varepsilon^2 \tilde{D}_{y} \phi^{n+1}_{\cdot, N} \nonumber \\
& \quad\quad\quad + A \dt \tilde{D}_{y}(\phi^{n+1}-\phi^n)_{\cdot,N}-B \dt D_x^2(\phi^{T,n+1}-\phi^{T,n}) , \label{2nd-scheme-CHDBC-6}  
\end{align}
where $\hat{\phi}^{n+1}=2\phi^n-\phi^{n-1}$, $\hat{\phi}^{T,n+1}=2\phi^{T,n}-\phi^{T,n-1}$, $\hat{\phi}^{B,n+1}=2\phi^{B,n}-\phi^{B,n-1}$. In fact, some constraints need to be imposed for the artificial regularization parameters $A$ and $B$ to ensure a discrete energy dissipation; see the details in Theorem \ref{2nd energy decay} below. The unique solvability and positivity-preserving analysis is established in the following theorem. Analogous to Theorem \ref{CHDBC-positivity}, the theoretical proof is based on the convexity and singularity analysis, and the technical details are skipped for the sake of brevity. 
\begin{theorem}  \label{2nd-CHDBC-positivity} 
Given $\phi^n,\ \phi^{n-1} \in \Cperx$, with $-1 < \phi^n_{i,j},\ \phi^{n-1}_{i,j} < 1$, $0 \le i, j \le N$,  and $\overline{\phi^n} =\overline{\phi^{n-1}} = \beta_0$, $\overline{\phi^{B,n}}^\Gamma=\overline{\phi^{B,n-1}}^\Gamma = \beta_{B,0}$, $\overline{\phi^{T,n}}^\Gamma=\overline{\phi^{T,n-1}}^\Gamma = \beta_{T,0}$, there exists a unique solution $\phi^{n+1} \in \Cperx$ to the numerical system \eqref{2nd-scheme-CHDBC-1}-\eqref{2nd-scheme-CHDBC-6}, with $-1 < \phi^{n+1}_{i,j} < 1$, $0 \le i, j \le N,$ and $\overline{\phi^{n+1}} = \beta_0$,  $\overline{\phi^{B,n+1}}^\Gamma  = \beta_{B,0}$, $\overline{\phi^{T,n+1}}^\Gamma  = \beta_{T,0}$. In particular, the solution is the unique minimizer and stationary point of the functional $\mJ_h^n$ given by
\begin{align}
    \mJ_h^n (\phi) & :=  \frac{1}{12\dt} \| 3\phi - 4\phi^n + \phi^{n-1} \|_{-1}^2 + (I(\phi),1) + \frac{2}{h}  (I(\phi^B),1)_\Gamma  + \frac{2}{h} (I(\phi^T),1)_\Gamma \nonumber \\
& \quad + \frac{1}{6h\dt } \| 3\phi^B - 4\phi^{B,n} + \phi^{B,n-1} \|_{-1, \Gamma}^2 + \frac{1}{6h\dt} \| 3\phi^T - 4\phi^{T,n} + \phi^{T,n-1} \|_{-1,\Gamma}^2    
	 \nonumber\\
& \quad  + \frac{\varepsilon^2}{2}(\phi, L_h\phi) + \frac{A\dt}{2}\left(\phi-\phi^n, L_h(\phi-\phi^n) \right) - \frac{\varepsilon\kappa}{h}(\phi^B, D_x^2\phi^B)_\Gamma - \frac{\varepsilon\kappa}{h} (\phi^T, D_x^2\phi^T)_\Gamma \nonumber \\
& \quad - \theta_0 \Big( (\hat{\phi}^{n+1},\phi) + \frac{2}{h} (\hat{\phi}^{B,n+1},\phi^B)_\Gamma + \frac{2}{h} (\hat{\phi}^{T,n+1},\phi^T)_\Gamma \Big) \nonumber \\
& \quad +\frac{B\dt}{h}\|D_x(\phi^{B}-\phi^{B,n})\|_{2,\Gamma}^2 +\frac{B\dt}{h} \|D_x(\phi^{T}-\phi^{T,n})\|_{2,\Gamma}^2,  \nonumber
\end{align}
with $I(\phi)$ is given in \eqref{function I}, over the admissible set
\begin{equation}
A_h :=  \left\{  \phi \in \Cperx \ \middle| \ -1 < \phi_{i,j}  < 1 , \ 0 \le j \le N , \ i\in\mathbb{Z}, \ \overline{\phi} = \beta_0, \ \overline{\phi_B}  = \beta_{B,0}, \ \overline{\phi_T}  = \beta_{T,0} \right\}. \nonumber
\end{equation} 
In other words, we have $\phi^{n+1} = \mathop{\mathrm{argmin}} \mJ_h^n(\phi)$. 
\end{theorem}

%\begin{proof}
%    The equivalence between the numerical solution and the stationary point lies in the fact that the Euler-Lagrange equation of functional $\mJ_h^n$ is identical with the numerical scheme \eqref{2nd-scheme-CHDBC-1}-\eqref{2nd-scheme-CHDBC-6}. See \cite{CH-DBC-2024} for the detailed calculation. 
%
%    Consider a smaller closed, convex domain: for $\delta \in (0,1/2)$, 
%\begin{equation}
%A_{h,\delta} :=  \left\{ \phi \in A_h  \ \middle| \ -1 + \delta \le \phi_{i,j}   \le 1-\delta, \ 0 \le j \le N, \ i\in\mathbb{Z}  \right\}  .
%\end{equation} 
%    Since $A_{h,\delta}$ is a convex, compact set, there exists a (not necessarily unique) minimizer of $F^n_h$ over $A_{h,\delta}$. The key point of our positivity analysis is that such a minimizer could not occur at one of the boundary points of $A_{h,\delta}$, if $\delta$ is sufficiently small.
%\end{proof}

In the following theorem, the energy stability is derived in terms of a modified energy functional with a few artificial stabilization terms.
\begin{theorem} \label{2nd energy decay}
Introduce a stabilized discrete energy
\begin{align}
\tilde{E}_h(\phi^{n+1},\phi^n)=E_h(\phi^{n+1})& +\frac{1}{4\dt}\| \phi^{n+1}-\phi^n \|_{-1}^2+\frac{\theta_0}{2}\|\phi^{n+1}-\phi^n\|_2^2 \nonumber \\
& +\frac{1}{4\dt}\| \phi^{T,n+1}-\phi^{T,n} \|_{-1,\Gamma}^2+\frac{\theta_0}{2}\|\phi^{T,n+1}-\phi^{T,n}\|_{2,\Gamma}^2 \nonumber \\
& +\frac{1}{4\dt}\| \phi^{B,n+1}-\phi^{B,n} \|_{-1,\Gamma}^2+\frac{\theta_0}{2}\|\phi^{B,n+1}-\phi^{B,n}\|_{2,\Gamma}^2 , 
\end{align}
where $E_h(\phi)$ is given by \eqref{CHDBC-discrete energy}. Under a stabilizer parameter constraint, $A,B \leq \frac{\theta_0^2}{16}$, for any time step size $\dt>0$, the numerical solution of \eqref{2nd-scheme-CHDBC-1}-\eqref{2nd-scheme-CHDBC-6} preserves a modified energy dissipation law
\begin{equation}
    \tilde{E}_h(\phi^{n+1},\phi^n) \leq \tilde{E}_h(\phi^{n},\phi^{n-1}).  \label{modified energy dissipation}
\end{equation}
\end{theorem}
\begin{proof}
By taking an inner product with \eqref{2nd-scheme-CHDBC-1} by $(-\Delta_h)^{-1}(\phi^{n+1}-\phi^n)$, an application of the summation-by-part formula gives 
\begin{align}
    (\frac{3\phi^{n+1}-4\phi^n+\phi^{n-1}}{2\dt}, (-\Delta_h)^{-1}(\phi^{n+1}-\phi^n)) + (\mu^{n+1},\phi^{n+1}-\phi^n) = 0. \label{2nd-energy-1}
\end{align}
The first term could be analyzed as follows: 
\begin{equation} 
\begin{aligned} 
	I:= & \Big( \frac{3\phi^{n+1}-4\phi^n+\phi^{n-1}}{2\dt}, (-\Delta_h)^{-1}(\phi^{n+1}-\phi^n) \Big) 
\\
  \ge & 
  \frac{5}{4\dt}\|\phi^{n+1}-\phi^n\|_{-1}^2-\frac{1}{4\dt}\|\phi^{n}-\phi^{n-1}\|_{-1}^2. 
\end{aligned} 
  \label{2nd-energy-2}
\end{equation}
Regarding the second term $II:=(\mu^{n+1},\phi^{n+1}-\phi^n)$, the convexity analysis implies that
\begin{align}
II &\geq (I(\phi^{n+1}),1)-(I(\phi^n),1)-\theta_0(2\phi^n-\phi^{n-1},\phi^{n+1}-\phi^n)-\varepsilon^2(\Delta_h\phi^{n+1},\phi^{n+1}-\phi^n) \nonumber \\
& \quad -A\dt(\Delta_h(\phi^{n+1}-\phi^n),\phi^{n+1}-\phi^n). \label{2nd-energy-3}
\end{align}
The following inequalities are valid: 
\begin{align}
-\theta_0(2\phi^n-\phi^{n-1},\phi^{n+1}-\phi^n) & \geq -\frac{\theta_0}{2}\|\phi^n-\phi^{n-1}\|_2^2-\frac{\theta_0}{2}(\|\phi^{n+1}\|_2^2-\|\phi^n\|_2^2), \label{2nd-energy-4}  \\
 -A\dt(\Delta_h(\phi^{n+1}-\phi^n),\phi^{n+1}-\phi^n)&=A\dt\|\nabla_h(\phi^{n+1}-\phi^n)\|_2^2 \nonumber \\
& \quad -A \dt (\tilde{D}_{y}(\phi^{n+1}-\phi^n)_{\cdot,N},\phi^{T,n+1}-\phi^{T,n})_\Gamma \nonumber \\
& \quad +A \dt (\tilde{D}_{y}(\phi^{n+1}-\phi^n)_{\cdot,0},\phi^{B,n+1}-\phi^{B,n})_\Gamma \label{2nd-energy-5}  \\
-\varepsilon^2(\Delta_h\phi^{n+1},\phi^{n+1}-\phi^n) & \geq \frac{\varepsilon^2}{2}(\|\nabla_h\phi^{n+1}\|_2^2-\|\nabla_h\phi^n\|_2^2)-\varepsilon^2(\tilde{D}_{y}\phi^{n+1}_{\cdot,N},\phi^{T,n+1}-\phi^{T,n})_\Gamma \nonumber \\
& \quad +\varepsilon^2(\tilde{D}_{y}\phi^{n+1}_{\cdot,0},\phi^{B,n+1}-\phi^{B,n})_\Gamma.\label{2nd-energy-6} 
\end{align}
The boundary inner product could be covered by taking an inner product with \eqref{2nd-scheme-CHDBC-4} by $(-D_x^2)^{-1}(\phi^{T,n+1}-\phi^{T,n})$. The following quantity is introduced: 
$$I_T=(\frac{3\phi^{T,n+1} - 4\phi^{T,n} + \phi^{T,n-1}}{2\dt},(-D_x^2)^{-1}(\phi^{T,n+1}-\phi^{T,n}))_\Gamma,$$ 
$$II_T=(\mu^{n+1}_T,\phi^{T,n+1}-\phi^{T,n})_\Gamma.$$ 
Similar estimates could be derived on the boundary section:
\begin{align}
&I_T \geq \frac{5}{4\dt}\|\phi^{T,n+1}-\phi^{T,n}\|_{-1,\Gamma}^2-\frac{1}{4\dt}\|\phi^{T,n}-\phi^{T,n-1}\|_{2,\Gamma}^2, \label{2nd-energy-b1} \\
&II_T \geq (I(\phi^{T,n+1}),1)-(I(\phi^{T,n}),1)-\frac{\theta_0}{2}\| \phi^{T,n}-\phi^{T,n-1} \|_{2,\Gamma}^2 -\frac{\theta_0}{2}(\|\phi^{T,n+1}\|_2^2-\|\phi^{T,n}\|_2^2) \nonumber \\
& \quad\quad +\frac{\varepsilon\kappa}{2}\|D_x\phi^{T,n+1}\|_{2,\Gamma}^2 -\frac{\varepsilon\kappa}{2}\|D_x\phi^{T,n}\|_{2,\Gamma}^2 + B\dt \|D_x(\phi^{T,n+1}-\phi^{T,n})\|_{2,\Gamma}^2 \nonumber \\
& \quad\quad +A \dt (\tilde{D}_{y}(\phi^{n+1}-\phi^n)_{\cdot,N},\phi^{T,n+1}-\phi^{T,n}) + \varepsilon^2(\tilde{D}_{y}\phi^{n+1}_{\cdot,N},\phi^{T,n+1}-\phi^{T,n})_\Gamma,  \label{2nd-energy-b2}
\end{align}
and
\begin{align}
&I_B \geq \frac{5}{4\dt}\|\phi^{B,n+1}-\phi^{B,n}\|_{-1,\Gamma}^2-\frac{1}{4\dt}\|\phi^{B,n}-\phi^{B,n-1}\|_{2,\Gamma}^2, \label{2nd-energy-b3} \\
&II_B \geq (I(\phi^{B,n+1}),1)-(I(\phi^{B,n}),1)-\frac{\theta_0}{2}\| \phi^{B,n}-\phi^{B,n-1} \|_{2,\Gamma}^2 -\frac{\theta_0}{2}(\|\phi^{B,n+1}\|_2^2-\|\phi^{B,n}\|_2^2) \nonumber \\
& \quad\quad +\frac{\varepsilon\kappa}{2}\|D_x\phi^{B,n+1}\|_{2,\Gamma}^2 -\frac{\varepsilon\kappa}{2}\|D_x\phi^{B,n}\|_{2,\Gamma}^2 + B\dt \|D_x(\phi^{B,n+1}-\phi^{B,n})\|_{2,\Gamma}^2 \nonumber \\
& \quad\quad -A \dt (\tilde{D}_{y}(\phi^{n+1}-\phi^n)_{\cdot,0},\phi^{B,n+1}-\phi^{B,n}) -\varepsilon^2(\tilde{D}_{y}\phi^{n+1}_{\cdot,0},\phi^{B,n+1}-\phi^{B,n})_\Gamma. \label{2nd-energy-b4}
\end{align}
A combination of the bulk and boundary inequalities leads to 
\begin{align}
E_h&(\phi^{n+1})-E_h(\phi^{n})-\frac{\theta_0}{2}\left( \|\phi^{n}-\phi^{n-1}\|_2^2+\|\phi^{T,n}-\phi^{T,n-1} \|_{2,\Gamma}^2+\|\phi^{B,n}-\phi^{B,n-1} \|_{2,\Gamma}^2 \right) \nonumber \\
&+A\dt\|\nabla_h(\phi^{n+1}-\phi^n)\|_2^2+B\dt \|D_x(\phi^{T,n+1}-\phi^{T,n})\|_{2,\Gamma}^2+B\dt \|D_x(\phi^{B,n+1}-\phi^{B,n})\|_{2,\Gamma}^2   \nonumber \\
&+\frac{5}{4\dt}\|\phi^{B,n+1}-\phi^{B,n}\|_{-1,\Gamma}^2-\frac{1}{4\dt}\|\phi^{B,n}-\phi^{B,n-1}\|_{2,\Gamma}^2 \nonumber \\
&+\frac{5}{4\dt}\|\phi^{T,n+1}-\phi^{T,n}\|_{-1,\Gamma}^2-\frac{1}{4\dt}\|\phi^{T,n}-\phi^{T,n-1}\|_{2,\Gamma}^2 \nonumber \\
&+\frac{5}{4\dt}\|\phi^{n+1}-\phi^{n}\|_{-1}^2-\frac{1}{4\dt}\|\phi^{n}-\phi^{n-1}\|_{2}^2 \leq 0.\label{2nd-energy-b5}
\end{align}
In turn, an application of Cauchy inequality yields the following estimates:
\begin{align}
    \frac{1}{\dt}\|\phi^{n+1}-\phi^{n}\|_{-1}^2+A\dt\|\nabla_h(\phi^{n+1}-\phi^{n})\|_2^2 &\geq 2\sqrt{A} \|\phi^{n+1}-\phi^{n}\|_2^2, \nonumber \\
   \frac{1}{\dt}\|\phi^{T,n+1}-\phi^{T,n}\|_{-1,\Gamma}^2+B\dt\|D_x(\phi^{T,n+1}-\phi^{T,n})\|_{2,\Gamma}^2 & \geq 2\sqrt{B} \|\phi^{T,n+1}-\phi^{T,n}\|_{2,\Gamma}^2, \label{2nd-energy-b6} \\
    \frac{1}{\dt}\|\phi^{B,n+1}-\phi^{B,n}\|_{-1,\Gamma}^2+B\dt\|D_x(\phi^{B,n+1}-\phi^{B,n})\|_{2,\Gamma}^2 &\geq 2\sqrt{B} \|\phi^{B,n+1}-\phi^{B,n}\|_{2,\Gamma}^2.\nonumber
\end{align}
Therefore, the desired inequality $\tilde{E}_h(\phi^{n+1},\phi^n) \leq \tilde{E}_h(\phi^{n},\phi^{n-1})$ is obtained, if $A,B\geq \frac{\theta_0^2}{16}$. 
\end{proof}
The same correction function defined in \eqref{sine auxiliary function 1}, as well as the Fourier projection in the $x$-direction, could be used in the convergence analysis. As mentioned in Remark \ref{order remark}, the associated auxiliary function would have a second order accuracy in space. The regularity requirement should be higher than class $\mathcal{R}$ in the first order scheme. Of course, the exact solution $\Phi$ is always assumed to have sufficient regularity. With the corrected solution $\hat{\Phi}_N=\mathcal{P}_x^N\Phi-\delta\Phi$, which contains the Fourier projection in the $x$-direction and the correction function, we define $\tilde{\phi}=\hat{\Phi}_N-\phi$ and $\tilde{\hat{\phi}}^{n+1}=2\tilde{\phi}^n-\tilde{\phi}^{n-1}$. The convergence result is stated in the following theorem. The proof will be brief, since the details are very similar to  the ones presented in Theorem \ref{Theorem 1}.
\begin{theorem}\label{2nd convergernce theorem}
Given smooth initial data $\Phi(\cdot,t=0)$, suppose the exact solution for CH equation with dynamical boundary condition is of sufficient regularity. Provided that $\dt$ and $h$ are sufficiently small, we have
\begin{equation}
\|\tilde{\phi}^{n+1}\|_{-1}+\|\tilde{\phi}^{n+1}\|_{-1,\Gamma} + \dt \sum_{k=1}^n \left( \varepsilon^2 \|\nabla_h \tilde{\phi}^{k}\|_2^2 + \varepsilon\kappa\|D_x \tilde{\phi}^{k}\|_{2,\Gamma}^2 \right) \leq C(\dt^2 + h^2) , 
  \label{2nd convergence theorem} 
\end{equation}
where $\|\tilde{\phi}^{n+1}\|_{-1,\Gamma}$ and $\|D_x \tilde{\phi}^{n+1}\|_{2,\Gamma}$ is given by \eqref{simplicity -1 norm} and \eqref{simplicity -1 norm 0}, and the positive constant $C$ is independent of $\dt$ and $h$.  
\end{theorem}
\begin{proof}
A careful consistency analysis indicates the following truncation error estimate:
\begin{align}
& \frac{3\tilde{\phi}^{n+1}-4\tilde{\phi}^n+\tilde{\phi}^{n-1}}{2\dt} = \Delta_h \tilde{\mu}^{n+1} + \tau^{n+1} , \label{2nd error eq-1} \\
& \tilde{\mu}^{n+1} = \ln(1+\hat{\Phi}_N^{n+1})-\ln(1-\hat{\Phi}_N^{n+1})-\ln(1+\phi^{n+1})+\ln(1-\phi^{n+1}) - \theta_0 \tilde{\hat{\phi}}^{n+1} , \nonumber \\
& \quad \quad\quad \quad - \varepsilon^2 \Delta_h \tilde{\phi}^{n+1}-A\dt\Delta_h(\tilde{\phi}^{n+1}-\tilde{\phi}^n) \label{2nd error eq-2}\\
& \frac{3\tilde{\phi}^{B,n+1}-4\tilde{\phi}^{B,n}+\tilde{\phi}^{B,n-1}}{2\dt} = D_x^2 \tilde{\mu}^{n+1}_B + \tau^{n+1}_B, \label{2nd error eq-3} \\
&\frac{3\tilde{\phi}^{T,n+1} - 4\tilde{\phi}^{T,n}+\tilde{\phi}^{B,n-1}}{2\dt} = D_x^2 \tilde{\mu}^{n+1}_T + \tau^{n+1}_T, \label{2nd error eq-4}	\\
& \tilde{\mu}^{n+1}_B = \ln(1+\hat{\Phi}_N^{B,n+1})-\ln(1-\hat{\Phi}_N^{B,n+1})-\ln(1+\phi^{B,n+1})+\ln(1-\phi^{B,n+1}) - \theta_0 \tilde{\hat{\phi}}^{B,n+1} \nonumber \\
& \quad \quad  - \varepsilon\kappa D_x^2 \tilde{\phi}^{B,n+1} - \varepsilon^2 \tilde{D}_{y} \tilde{\phi}^{n+1}_{\cdot, 0}- A \dt \tilde{D}_{y}(\tilde{\phi}^{n+1} - \tilde{\phi}^n)_{\cdot,0}-B \dt D_x^2(\tilde{\phi}^{B,n+1}-\tilde{\phi}^{B,n}), \label{2nd error eq-5}\\
& \tilde{\mu}^{n+1}_T = \ln(1+\hat{\Phi}_N^{T,n+1})-\ln(1-\hat{\Phi}_N^{T,n+1})-\ln(1+\phi^{T,n+1})+\ln(1-\phi^{T,n+1}) - \theta_0 \tilde{\hat{\phi}}^{T,n+1} \nonumber \\
& \quad \quad  - \varepsilon\kappa D_x^2 \tilde{\phi}^{T,n+1} + \varepsilon^2 \tilde{D}_{y} \tilde{\phi}^{n+1}_{\cdot, N}+ A \dt \tilde{D}_{y}(\tilde{\phi}^{n+1}-\tilde{\phi}^n)_{\cdot,N}-B \dt D_x^2(\tilde{\phi}^{T,n+1}-\tilde{\phi}^{T,n}) , \label{2nd error eq-6}
\end{align}
with truncation error accuracy order
\begin{equation}
\|\tau^{n+1}\|_2,\ \|\tau^{n+1}_B\|_{2,\Gamma},\ \|\tau^{n+1}_T\|_{2,\Gamma} \leq O(\dt^2+h^2). \label{2nd truncation}
\end{equation}
Taking an inner product with \eqref{2nd error eq-1} by $2(-\Delta_h)^{-1}\tilde{\phi}^{n+1}$ gives
\begin{align}
\frac{1}{2\dt}&\left(\|\tilde{\phi}^{n+1}\|_{-1}^2-\|\tilde{\phi}^n\|_{-1}^2+\|2 \tilde{\phi}^{n+1}-\tilde{\phi}^n\|_{-1}^2-\|2 \tilde{\phi}^n-\tilde{\phi}^{n-1}\|_{-1}^2\right)  \nonumber\\
& - 2\varepsilon^2(\Delta_h \tilde{\phi}^{n+1},\tilde{\phi}^{n+1}) - 2A\dt(\Delta_h (\tilde{\phi}^{n+1}-\tilde{\phi}^{n}),\tilde{\phi}^{n+1})\leq \nonumber \\
& \qquad\qquad\qquad\qquad  2\theta_0(\tilde{\hat{\phi}}^{n+1},\tilde{\phi}^{n+1}) + \|\tau^{n+1}\|_{-1}^2+ \|\tilde{\phi}^{n+1}\|_{-1}^2 , \label{2nd error 2.1}
\end{align}
in which the convexity inequalities \eqref{convergence 1.1}-\eqref{convergence 1.2} have been applied. The concave expansive error term could be bounded in a straightforward manner:  
\begin{align}
2\theta_0(\tilde{\hat{\phi}}^{n+1},\tilde{\phi}^{n+1})& \leq 2\theta_0\|\nabla_h\tilde{\phi}^{n+1}\|_2 \cdot \|\tilde{\hat{\phi}}^{n+1}\|_{-1} \leq \frac{\theta_0^2}{\varepsilon^2}\|\tilde{\hat{\phi}}^{n+1}\|_{-1}^2+\varepsilon^2\|\nabla_h \tilde{\phi}^{n+1}\|_2^2 \nonumber \\
& \leq \frac{\theta_0^2}{\varepsilon^2}(4\|\tilde{\phi}^n\|_{-1}^2+\|\tilde{\phi}^{n-1}\|_{-1}^2) + \varepsilon^2\|\nabla_h \tilde{\phi}^{n+1}\|_2^2. \label{2nd error 2.2}
\end{align}
The diffusion terms, including the artificial diffusion part, are similarly treated: 
\begin{align}
& -2\varepsilon^2(\Delta_h \tilde{\phi}^{n+1},\tilde{\phi}^{n+1}) = 2\varepsilon^2 \|\nabla_h \tilde{\phi}^{n+1}\|_2^2-2\varepsilon^2(\tilde{D}_{y}\tilde{\phi}^{n+1}_{\cdot,N},\tilde{\phi}^{n+1}_{\cdot,N})_\Gamma+2\varepsilon^2(\tilde{D}_{y}\tilde{\phi}^{n+1}_{\cdot,0},\tilde{\phi}^{n+1}_{\cdot,0})_\Gamma, \label{2nd error 2.3} \\
& - 2A\dt(\Delta_h (\tilde{\phi}^{n+1}-\tilde{\phi}^{n}),\tilde{\phi}^{n+1}) \geq A \dt ( \|\nabla_h\tilde{\phi}^{n+1}\|_2^2-\|\nabla_h\tilde{\phi}^n\|_2^2 ) \nonumber \\
& \quad\quad\quad\quad -2A\dt (\tilde{D}_{y}(\tilde{\phi}^{n+1}-\tilde{\phi}^n)_{\cdot,N},\tilde{\phi}^{n+1}_{\cdot,N})_\Gamma+2A\dt (\tilde{D}_{y}(\tilde{\phi}^{n+1}-\tilde{\phi}^n)_{\cdot,0},\tilde{\phi}^{n+1}_{\cdot,0})_\Gamma. \label{2nd error 2.3.1}
\end{align}
The boundary part is covered by the boundary inner product with \eqref{2nd error eq-3} by $2(-D_x^2)^{-1}\tilde{\phi}^{B,n+1}$ (and with \eqref{2nd error eq-4} by $2(-D_x^2)^{-1}\tilde{\phi}^{T,n+1}$): 
\begin{align}
\frac{1}{2\dt}&\left(\|\tilde{\phi}^{B,n+1}\|_{-1,\Gamma}^2-\|\tilde{\phi}^{B,n}\|_{-1,\Gamma}^2+\|2 \tilde{\phi}^{B,n+1}-\tilde{\phi}^{B,n}\|_{-1,\Gamma}^2-\|2 \tilde{\phi}^{B,n}-\tilde{\phi}^{B,n-1}\|_{-1,\Gamma}^2\right)  \nonumber\\
&+ 2\varepsilon\kappa\|D_x \tilde{\phi}^{B,n+1} \|_{2,\Gamma} + B\dt(\|D_x \tilde{\phi}^{B,n+1}\|_{2,\Gamma} - \|D_x \tilde{\phi}^{B,n}\|_{2,\Gamma}) \nonumber\\
& \leq 2(\varepsilon^2 \tilde{D}_{y} \tilde{\phi}^{n+1}_{\cdot, 0} + A\dt \tilde{D}_{y}(\tilde{\phi}^{n+1}-\tilde{\phi}^n)_{\cdot,0}, \tilde{\phi}^{B,n+1})_\Gamma + 2\theta_0(\tilde{\hat{\phi}}^{B,n+1}, \tilde{\phi}^{B,n+1})_\Gamma  \nonumber\\
& \quad\quad + \|\tau^{n+1}_B\|_{-1,\Gamma}^2+ \|\tilde{\phi}^{B,n+1}\|_{-1,\Gamma}^2 , 
\end{align}
and
\begin{align}
	\frac{1}{2\dt}&\left(\|\tilde{\phi}^{T,n+1}\|_{-1,\Gamma}^2-\|\tilde{\phi}^{T,n}\|_{-1,\Gamma}^2+\|2 \tilde{\phi}^{T,n+1}-\tilde{\phi}^{T,n}\|_{-1,\Gamma}^2-\|2 \tilde{\phi}^{T,n}-\tilde{\phi}^{T,n-1}\|_{-1,\Gamma}^2\right)  \nonumber\\
	&+ 2\varepsilon\kappa\|D_x \tilde{\phi}^{T,n+1} \|_{2,\Gamma} + B\dt(\|D_x \tilde{\phi}^{T,n+1}\|_{2,\Gamma} - \|D_x \tilde{\phi}^{T,n}\|_{2,\Gamma}) \nonumber\\
	& \leq -2(\varepsilon^2 \tilde{D}_{y} \tilde{\phi}^{n+1}_{\cdot, N} + A\dt \tilde{D}_{y}(\tilde{\phi}^{n+1}-\tilde{\phi}^n)_{\cdot,N}, \tilde{\phi}^{T,n+1})_\Gamma + 2\theta_0(\tilde{\hat{\phi}}^{T,n+1}, \tilde{\phi}^{T,n+1})_\Gamma  \nonumber\\
	& \quad\quad + \|\tau^{n+1}_T\|_{-1,\Gamma}^2+ \|\tilde{\phi}^{T,n+1}\|_{-1,\Gamma}^2 . 
\end{align}
Notice that the following quantity has been introduced in the derivation: 
\begin{equation}
\|D_x \phi\|_{2,\Gamma}^2:=\|\phi^T\|_{2,\Gamma}^2+\|\phi^B\|_{2,\Gamma}^2, \quad \forall \phi \in {\mathcal C}_{\rm per}^x.
\end{equation}
Similar to the estimate \eqref{2nd error 2.2}, the expansion error has the following bound: 
\begin{align}
& 2\theta_0(\tilde{\hat{\phi}}^{B,n+1},\tilde{\phi}^{B,n+1})_\Gamma \leq \frac{\theta_0^2}{\varepsilon\kappa}(4\|\tilde{\phi}^{B,n}\|_{-1,\Gamma}^2+\|\tilde{\phi}^{B,n-1}\|_{-1,\Gamma}^2) + \varepsilon\kappa\|D_x \tilde{\phi}^{B,n+1}\|_{2,\Gamma}^2, \\
& 2\theta_0(\tilde{\hat{\phi}}^{T,n+1},\tilde{\phi}^{T,n+1})_\Gamma \leq \frac{\theta_0^2}{\varepsilon\kappa}(4\|\tilde{\phi}^{T,n}\|_{-1,\Gamma}^2+\|\tilde{\phi}^{T,n-1}\|_{-1,\Gamma}^2) + \varepsilon\kappa\|D_x \tilde{\phi}^{T,n+1}\|_{2,\Gamma}^2.
\end{align}
As a consequence, we arrive at 
\begin{align}
& \frac{1}{2\dt}\left(\|\tilde{\phi}^{n+1}\|_{-1}^2-\|\tilde{\phi}^n\|_{-1}^2+\|2 \tilde{\phi}^{n+1}-\tilde{\phi}^n\|_{-1}^2-\|2 \tilde{\phi}^n-\tilde{\phi}^{n-1}\|_{-1}^2\right)  \nonumber \\
& +\frac{1}{2\dt}\left(\|\tilde{\phi}^{n+1}\|_{-1,\Gamma}^2-\|\tilde{\phi}^n\|_{-1, \Gamma}^2+\|2 \tilde{\phi}^{n+1}-\tilde{\phi}^n\|_{-1, \Gamma}^2-\|2 \tilde{\phi}^n-\tilde{\phi}^{n-1}\|_{-1, \Gamma}^2\right) \nonumber \\
& + A \dt \left( \|\nabla_h\tilde{\phi}^{n+1}\|_2^2-\|\nabla_h\tilde{\phi}^n\|_2^2 \right) + B \dt\left( \|D_x\tilde{\phi}^{n+1}\|_{2,\Gamma}^2-\|D_x\tilde{\phi}^n\|_{2,\Gamma}^2 \right) \nonumber \\
& + \varepsilon^2 \|\nabla_h \tilde{\phi}^{n+1}\|_2^2 + \varepsilon\kappa\|D_x \tilde{\phi}^{B,n+1}\|_{2,\Gamma}^2 + \varepsilon\kappa\|D_x \tilde{\phi}^{T,n+1}\|_{2,\Gamma}^2
\leq \nonumber \\
& \quad\quad\quad C_{\varepsilon,\kappa}\left(4\|\tilde{\phi}^n\|_{-1}^2+4\|\tilde{\phi}^n\|_{-1,\Gamma}^2+\|\tilde{\phi}^{n-1}\|_{-1}^2+\|\tilde{\phi}^{n-1}\|_{-1,\Gamma}^2\right)+ \|\tilde{\phi}^{n+1}\|_{-1}^2+ \|\tilde{\phi}^{n+1}\|_{-1,\Gamma}^2 \nonumber \\
& \quad\quad\quad +\|\tau^{n+1}\|_{-1}^2+\|\tau^{T,n+1}\|_{-1,\Gamma}^2+\|\tau^{B,n+1}\|_{-1,\Gamma}^2 , \label{2nd error 2.4}
\end{align}
where $C_{\varepsilon,\kappa} = \max \{ \frac{\theta_0^2}{\varepsilon\kappa}, \frac{\theta_0^2}{\varepsilon^2} \}$.
Finally, an application of a discrete Gronwall inequality results in the desired convergence estimate \eqref{2nd convergence theorem}. This completes the proof.
\end{proof}
\begin{rem}
In fact, the convergence in standard $H^{-1}$ space after correction is equivalent to a modified $H^{-1}$ norm without correction:
$$
\| \phi \|_{-1,m} := \| \phi-\overline{\phi} \|_{-1} + |\overline{\phi}|, \quad \phi \in L^2.
$$
This new norm no longer requires the mean-zero property and is equivalent to the standard $H^{-1}$ norm for mean-zero functions. Meanwhile, it also fulfills the embedding theory. Applying certain modifications to the norm to handle complex properties may be a new theoretical approach in future works.
\end{rem}

\section{Conclusions}  \label{sec:conclusion} 
In this paper, we establish an $H^{-1}$-convergence analysis of two numerical schemes to the CH equation with dynamical boundary condition, including both the first order convex splitting method and the second order accurate one. An explicitly defined auxiliary function is designed to reach the discrete mass conservation of the exact solution, so that the discrete $H^{-1}$ norm of the numerical error function is well-defined. This correction method turns out to be an alternate approach, since it replaces the Fourier projection without loss of accuracy. The additional auxiliary function is in the same order as the spatial truncation error and is independent with time step, so that it does not affect the convergence rate. In turn, this methodology could be generalized to problems with different boundary conditions. The smoothness of the auxiliary function lays the foundation of the theoretical analysis, and this approach could be combined with various methods and problems. 

\section*{Acknowledgments}
C. Wang is partially supported by the National Science Foundation, DMS-2012269, DMS-2309548 and Z.R. Zhang is partially supported by the NSFC, PR China No. 11871105 and 12231003. Y.Z. Guo thanks the Hong Kong Polytechnic University for the generous support.
%\newpage
\bibliographystyle{plain}
\bibliography{draft_ref}

%\section*{Declarations of interest} 

%None 

\end{document}